\documentclass[12pt,a4paper]{amsart}

\usepackage[T1]{fontenc}
\usepackage{amssymb,amsmath,latexsym,amsthm,paralist,a4,mathrsfs,dsfont}
\usepackage{tikz-cd}
\usetikzlibrary{decorations.markings}

\makeatletter
\tikzcdset{
  open/.code     = {\tikzcdset{hook, circled};},
  closed/.code   = {\tikzcdset{hook, slashed};},
  open'/.code    = {\tikzcdset{hook', circled};},
  closed'/.code  = {\tikzcdset{hook', slashed};},
  circled/.code  = {\tikzcdset{markwith = {\draw (0,0) circle (.375ex);}};},
  slashed/.code  = {\tikzcdset{markwith = {\draw[-] (-.4ex,-.4ex) -- (.4ex,.4ex);}};},
  markwith/.code ={
    \pgfutil@ifundefined%
    {tikz@library@decorations.markings@loaded}%
    {\pgfutil@packageerror{tikz-cd}{You need to say %
      \string\usetikzlibrary{decorations.markings} to use arrows with markings}{}}{}%
    \pgfkeysalso{/tikz/postaction = {
      /tikz/decorate,
      /tikz/decoration={markings, mark = at position 0.5 with {#1}}}
    }
  },
}
\makeatother

\usepackage[headings]{fullpage}
\usetikzlibrary{babel}
\usetikzlibrary{positioning,calc}
\usepackage{varioref}
\usepackage[pdfpagelabels, pdftex]{hyperref}
\hypersetup{
  pdftitle={Logarithmic differentials on discretely ringed adic spaces},
  pdfauthor={Katharina H\"{u}bner},
  pdfsubject={},
  pdfkeywords={},
  colorlinks=true,    
  linkcolor=red,     
  citecolor=blue,     
  filecolor=blue,      
  urlcolor=blue,       
  breaklinks=true,
  bookmarksopen=true,
  bookmarksnumbered=true,
  pdfpagemode=UseOutlines,
  plainpages=false,
  unicode=true
  }
\usepackage[capitalise]{cleveref}
\usepackage{changes}  
\usepackage{bbold} 
\newcommand{\mapsfrom}{\mathrel{\reflectbox{\ensuremath{\mapsto}}}}

\newcommand{\id}{\mathrm{id}}

\newcommand{\Gl}{\mathrm{Gl}}

\newcommand{\Hom}{\mathrm{Hom}}

\newcommand{\Log}{\mathrm{Log}}

\newcommand{\pre}{\mathrm{pre}}
\newcommand{\Mon}{\mathrm{Mon}}
\newcommand{\Set}{\mathrm{Set}}
\newcommand{\Ring}{\mathrm{Ring}}
\newcommand{\Forget}{\mathrm{Forget}}
\newcommand{\Free}{\mathrm{Free}}
\newcommand{\Ho}{\mathrm{Ho}}
\newcommand{\straff}{\mathrm{straff}}

\newcommand{\m}{\mathfrak{m}}

\newcommand{\Z}{\mathds Z}
\newcommand{\N}{\mathds N}
\newcommand{\Q}{\mathds Q}

\newcommand{\R}{\mathds R}
\newcommand{\bC}{\mathds C}
\newcommand{\LL}{\mathds L}

\newcommand{\Top}{\mathrm{top}}
\newcommand{\et}{\textit{\'et}}
\newcommand{\sh}{\mathrm{sh}}

\newcommand{\Mod}{\mathrm{Mod}}

\newcommand{\ZZ}{\mathbb{Z}}
\renewcommand{\P}{{\mathbb P}}

\newcommand{\gp}{\mathrm{gp}}

\newcommand{\cC}{{\mathcal C}}

\newcommand{\cF}{{\mathcal F}}
\newcommand{\cG}{{\mathcal G}}

\newcommand{\cO}{{\mathcal O}}

\newcommand{\cR}{{\mathcal R}}

\newcommand{\cU}{{\mathcal U}}
\newcommand{\cV}{{\mathcal V}}

\newcommand{\cX}{{\mathcal X}}
\newcommand{\cY}{{\mathcal Y}}

\newcommand{\liso}{\mathrel{\hbox{$\longrightarrow$} \kern-2.4ex\lower-1ex\hbox{$\scriptstyle\sim$}\kern1.7ex}}

\newcommand{\set}{\textit{s\'et}}

\makeatletter
\newcommand{\zerounderset}[3][\mathord]{%
  #1{\vtop{
    \let\\\cr
    \baselineskip\z@skip\lineskip.25ex
    \ialign{\hidewidth$##$\hidewidth\crcr
      \omit$#3$\cr
      #2\crcr
    }%
  }}%
}
\makeatother

\newtheoremstyle{alexthm}
  {}
  {}
  {\sl }
  {}
  {\bf}
  {.}
  {.5em}
  {}
\theoremstyle{alexthm}

\newtheorem{theorem}{Theorem}[section]
\newtheorem*{theorem*}{Theorem}
\newtheorem{corollary}[theorem]{Corollary}
\newtheorem{proposition}[theorem]{Proposition}
\newtheorem{lemma}[theorem]{Lemma}
\newtheorem*{lemma*}{Lemma}

\newtheoremstyle{alexdef}
  {}
  {}
  {\rm }
  {}
  {\bf}
  {.}
  {.5em}
  {}
\theoremstyle{alexdef}
\newtheorem*{example*}{Example}
\newtheorem{example}[theorem]{Example}

\newtheorem{remark}[theorem]{Remark}

\newtheorem{definition}[theorem]{Definition}

\DeclareMathOperator{\Spec}{\mathrm{Spec}}

\DeclareMathOperator{\Spa}{\mathrm{Spa}}
\DeclareMathOperator{\supp}{\mathrm{supp}}

\DeclareMathOperator*{\colim}{colim}

\DeclareMathOperator{\ad}{\textit{ad}}

\definecolor{darklimegreen}{RGB}{31,142,8}

\begin{document}

\hfuzz=4pt
\title{Logarithmic differentials on discretely ringed adic spaces}
\author{Katharina H\"{u}bner}
\email{khuebner@mathi.uni-heidelberg.de}
\date{\today}
\address{Einstein Institute of Mathematics, Hebrew University, Jerusalem}
\thanks{This research is partly supported by ERC Consolidator Grant 770922 - BirNonArchGeom.
 Moreover the author acknowledges support by Deutsche Forschungsgemeinschaft  (DFG) through the Collaborative Research Centre TRR 326 "Geometry and Arithmetic of Uniformized Structures", project number 444845124.}

\begin{abstract}
 On a smooth discretely ringed adic space~$\cX$ over a field~$k$ we define a subsheaf $\Omega_{\cX}^+$ of the sheaf of differentials $\Omega_{\cX}$.
 It is defined in a similar way as the subsheaf $\cO^+_{\cX}$ of~$\cO_{\cX}$ using K\"ahler seminorms on $\Omega_{\cX}$.
 We give a description of $\Omega^+_{\cX}$ in terms of logarithmic differentials.
 If~$\cX$ is of the form $\Spa(X,\bar{X})$ for a scheme~$\bar{X}$ and an open subscheme~$X$ such that the corresponding log structure on~$\bar{X}$ is smooth, we show that~$\Omega^+_{\cX}(\cX)$ is isomorphic to the logarithmic differentials of $(X,\bar{X})$.
\end{abstract}

\maketitle
\tableofcontents

\section{Introduction}
Consider a discretely ringed adic space~$\cX$ over a valued field~$(k,k^+)$.
Here, discretely ringed means that~$\cX$ is locally isomorphic to the spectrum of a Huber pair $(A,A^+)$, where~$A$ and~$A^+$ carry the discrete topology.
The structure sheaf $\cO_{\cX}$ contains a natural subsheaf~$\cO^+_{\cX}$, the subsheaf of sections with germs of absolute valuation less or equal to one.
One might ask for a similar partner~$\Omega^+$ for the sheaf of differentials $\Omega_{\cX} = \Omega^1_{\cX/k}$.
It should be a subsheaf of~$\Omega:= \Omega^1_{\cX}$ defined by a condition $|\omega_x| \le 1$ for suitable $\cO_{\cX,x}$-seminorms $|\cdot|$ on the stalks~$\Omega_{\cX,x}$ for every point $x \in \cX$.
Such a sheaf~$\Omega^+$ will be useful for investigating cohomological purity for $p$-torsion sheaves in characteristic $p > 0$.
As explained in \cite{Milne86}, \S~2 the logarithmic de Rham sheaves $\nu(r)$ play a crucial role in cohomological purity.
They are defined by an exact sequence
\[
 0 \to \nu(r) \to \Omega^r_{d=0} \overset{C-1}{\longrightarrow} \Omega^r \to 0,
\]
in the \'etale topology.
Here, ``$d=0$'' refers to closed forms and~$C$ denotes the Cartier operator (see \cite{Milne76}, \S~1).
However, we expect purity to hold only for the tame topology (see \cite{HueAd} for the definition) and the above sequence is not exact in the tame topology.
We hope to solve this problem by replacing~$\Omega^r$ with $\Omega^{r,+}$.
This will be subject to future investigations.

In this article we construct a sheaf~$\Omega^+$ as above using the K\"ahler seminorms (cf. \cite{Tem16}, \S~4.1 for the real valued case) on the stalks~$\Omega_x$ defined by
\[
 |\omega|_{\Omega} := \inf_{\omega = \sum_i f_i dg_i} \max_i \{|f_i| \cdot |g_i|\},
\]
where the infimum is taken over all representations of~$\omega$ as a finite sum $\sum_i f_i dg_i$ (see \cref{section_Kaehler_local}).
However, we need to take care of where we take the infimum.
The value group~$\Gamma$ of the valuation on~$\cO_x$ is not complete, in general.
We use the new concept of rangers (studied in joint ongoing work of the author with Michael Temkin) to present a construction that serves as completion (see \cref{section_rangers}).
As a further preparation for the definition of the Kähler seminorm in \cref{section_Kaehler_local} we study seminorms taking values in rangers in \cref{section_seminorms}.
In \cref{section_Kaehler_adic} we prove that~$\Omega^+$ is indeed a sheaf on~$\cX$.
In fact, it is even a sheaf on the tame site~$\cX_t$ of~$\cX$ but not on the \'etale site.

It turns out that~$\Omega^+$ has a description in terms of logarithmic differentials.
After a preliminary section on the logarithmic cotangent complex (see \cref{section_log_cotangent}), we study logarithmic differentials in \cref{section_log_diff_adic}.
Let us specify the connection of logarithmic differentials with~$\Omega^+$.
For a Huber pair $(A,A^+)$ over~$k$ such that~$A$ is a localization of~$A^+$, we equip~$A^+$ with the log structure $(A^+ \cap A^\times \to A^+)$ on~$A^+$.
The corresponding log structure on $\Spec A^+$ is the compactifying log structure associated with the open embedding $\Spec A \hookrightarrow \Spec A^+$.
The corresponding logarithmic differentials~$\Omega^{\log}_{(A,A^+)}$ define a presheaf~$\Omega^{\log}$ but not a sheaf.
We prove that the sheafification of~$\Omega^{\log}$ is~$\Omega^+$ in \cref{section_Kaehler_adic}.
An important input is that for a local Huber pair $(A,A^+)$ over~$k$ the logarithmic differentials~$\Omega^{\log}_{(A,A^+)}$ imbed into~$\Omega_A$.
For this we need to put some restrictions on $(k,k^+)$.
To sum up we have the following theorem (see \cref{Kaeher_sheaf} and \cref{sheafification_logarithmic}):

\begin{theorem}
  Let~$\cX$ be a discretely ringed adic space over $(k,k^+)$.
 \begin{enumerate}
  \item $\Omega^+$ is a sheaf on the tame site~$\cX_t$.
  \item	Assume that either the residue characteristic of~$k^+$ is zero, $k$ is algebraically closed, or~$k = k^+$ is perfect.
		Then~$\Omega^+$ is the Zariski sheafification of the presheaf of logarithmic differentials.
 \end{enumerate}
\end{theorem}

The last section is dedicated to a study of logarithmic differentials on adic spaces of the form $\Spa(Y,\bar{Y})$, where~$\bar{Y}$ is a scheme over the field~$k$ and~$Y$ is an open subscheme such that the associated log structure on~$\bar{Y}$ is log smooth.
We call pairs $(Y,\bar{Y})$ of this type log-smooth pairs over~$k$.
The main result (\cref{main_theorem}) is the following

\begin{theorem}
 Let $(Y,\bar{Y})$ be log smooth.
 Then
 \[
  \Omega^+(\Spa(Y,\bar{Y})) \cong \Omega^{\log}(Y,\bar{Y}),
 \]
 where~$\Omega^{\log}$ on the right hand side is the sheaf of logarithmic differentials on the log scheme associated with $(Y,\bar{Y})$.
\end{theorem}

The crucial point is that on the adic space $\Spa(Y,\bar{Y})$ we do not need to sheafify~$\Omega^{\log}$ in order to compute the global sections of~$\Omega^+$.
This makes~$\Omega^+$ a lot more accessible and it is possible to use the theory of logarithmic differentials on log schemes to investigate~$\Omega^+$.
We also want to stress that the above isomorphism is obtained without assuming resolution of singularities.
The proof relies on the theory of unramified sheaves (see \cref{section_unramified_sheaves}), a notion adapted from \cite{Morel12}, and techniques similar to the ones applied in \cite{HKK17} for studying cdh differentials.

\medskip\noindent
\emph{Acknowledgement:} The author wants to thank Michael Temkin for drawing her attention to K\"ahler seminorms.
Moreover, many thanks go to Steffen Sagave for his help with the logarithmic cotangent complex.

\section{The logarithmic cotangent complex} \label{section_log_cotangent}

On a discretely ringed adic space~$\cX$ we want to study a subsheaf~$\Omega_\cX^+$ of the sheaf of differentials~$\Omega_\cX$ which is closely related to logarithmic differentials.
For future work it will be important to us that this is also a sheaf for the tame topology (not only on the topological space~$\cX$).
For this reason we need to study the log cotangent complex of a tame extension of valuation rings.
The reader not interested in the resulting technical sections~\ref{section_log_cotangent} and~\ref{section_unramified_tame} can skip them and jump to \cref{section_log_diff_adic}.

In \cite{Ols05} Olsson describes two approaches for a logarithmic cotangent complex.
His own construction using log stacks has the advantage that it is trivial for log smooth morphisms.
However, transitivity triangles only exist under certain conditions and the construction only works for fine log schemes, i.e. under strong finiteness conditions that are not satisfied in our situation.
Gabber's version described in \cite{Ols05}, \S 8 is more functorial but it has the disadvantage that it is not trivial for all log smooth morphisms.
We will use Gabber's log cotangent complex and compare it in special situations to Olsson's in order to make explicit computations.
Slightly more generally we will define the log cotangent complex for simplicial prelog rings as described for instance in \cite{Bha12}, \S 5 or \cite{SSV16}, \S 4.

Let us start with reviewing some definitions.
Recall that a prelog ring is a ring~$R$ and a (commutative) monoid~$M$ together with a homomorphism of monoids $M \to R$, where~$R$ is considered as a monoid with its multiplicative structure.
A log ring is a prelog ring $\iota: M \to R$ inducing an isomorphism $\iota^{-1}(R^{\times}) \to R^{\times}$.
The inclusion of the category of log rings into prelog rings has a left adjoint, logification (see \cite{Ogus18}, Chapter~II, Proposition~1.1.5)
We write $(M^a \to R)$ or $(M \to R)^a$ for the logification of $(M \to R)$.

Denote by $\Set$, $\Mon$, $\Ring$, and $\Log\Ring^{\pre}$ the categories of sets, monoids, rings, and prelog rings.
We write $s\Set$, $s\Mon$, $s\Ring$, and $s\Log\Ring^{\pre}$ for the respective categories of simplicial objects.
We endow $s\Set$ with the standard model structure, i.e. the weak equivalences are the maps inducing a weak homotopy equivalence on geometric realizations and the fibrations are the Kan fibrations.
Defining the (trivial) fibrations to be the homomorphisms that are (trivial) fibrations on the underlying category of simplicial sets, we obtain a closed model structure on $s\Ring$ and $s\Mon$ (see \cite{Bha12}, \S 4).
Now consider the forgetful functor
\[
 \Forget^{s\Log\Ring^{\pre}}_{s\Mon \times s\Ring} : s\Log\Ring^{\pre} \longrightarrow s\Mon \times s\Ring
\]
mapping $(M \to A)$ to $(M,A)$.
By \cite{SSV16}, Proposition~3.3 there is a projective proper simplicial cellular model structure on $s\Log\Ring^{\pre}$ whose fibrations and weak equivalences are the maps that are mapped to fibrations and weak equivalences, respectively, under $\Forget^{s\Log\Ring^{\pre}}_{s\Mon \times s\Ring}$.
With respect to this model structure $\Forget^{s\Log\Ring^{\pre}}_{s\Mon \times s\Ring}$ is a left and right Quillen functor (\cite{Bha12}, Propositions~5.3 and~5.5).
Its left adjoint is the functor $\Free^{s\Mon \times s\Ring}_{s\Log\Ring^{\pre}}$ mapping $(M,A)$ to $(M \to A[M])$.

For a homomorphism $(M \to A) \to (N \to B)$ of simplicial prelog rings we write $s\Log\Ring^{\pre}_{(M \to A) // (N \to B)}$ for the category of simplicial $(M \to A)$-algebras over $(N \to B)$.
It inherits a model structure from $s\Log\Ring^{\pre}$.
Consider the functor
\begin{align*}
 \Omega: s\Log\Ring^{\pre}_{(M \to A) // (N \to B)} &\to \Mod_{B}	\\
 (L \to C) &\mapsto \Omega^1_{(L \to C)/(N \to B)} \otimes_{C} B
\end{align*}
where~$\Omega^1$ is defined by applying to each level the functor of log K\"ahler differentials (see \cite{Ogus18}, Chapter~IV, Proposition~1.1.2; note that a log ring in loc. cit. is what we here call a prelog ring).
Being a left Quillen functor (\cite{SSV16}, Lemma~4.6), it has a left derived functor
\[
 L\Omega : \Ho(s\Log\Ring^{\pre}_{(M \to A) // (N \to B)}) \to \Ho(\Mod_{B})
\]
on the respective homotopy categories.
The image of $(N \to B)$ under~$L\Omega$ is called the \emph{cotangent complex} of $(N \to B)$ and denoted $\LL_{(M \to A) / (N \to B)}$.
For a homomorphism $(M \to A) \to (N \to B)$ of discrete log rings it can be computed as follows.
For shortness write $F:= \Forget^{\Log\Ring^{\pre}}_{\Mon \times \Ring}$ and $G:= \Free^{\Mon \times \Ring}_{\Log\Ring^{\pre}}$ (the discrete versions of the above considered functors).
We have a canonical free resolution
\begin{equation} \label{canonical_resolution}
 \begin{tikzcd}
  \ldots	\ar[r,shift left=1.5]	\ar[r,shift right=1.5]	\ar[r,shift left=4.5]	\ar[r,shift right=4.5]	& GFGF(N \to B)	\ar[l]	\ar[l,shift left=3]	\ar[l,shift right=3]	\ar[r]	\ar[r,shift left=3]	\ar[r,shift right=3]	& GF(N \to B)	\ar[l,shift left=1.5]	\ar[l,shift right=1.5]	\ar[r,shift left=1.5]	\ar[r,shift right=1.5]	& (N \to B)	\ar[l],
 \end{tikzcd}
\end{equation}
which we denote by $P_{\bullet} \to (N \to B)$.
Then $\LL_{(M \to A) / (N \to B)}$ is represented by $\Omega(P_{\bullet})$.
In particular, we recover Gabber's definition (\cite{Ols05}, Definition~8.5).

The cotangent complex has the following two important properties (see \cite{SSV16}, Proposition~4.12).

\begin{proposition} \label{transitivity_BC}
 \begin{enumerate}[(i)]
  \item	\emph{Transitivity.}
		Let $(M \to A) \to (N \to B) \to (K \to C)$ be maps of simplicial prelog rings.
		Then there is a homotopy cofiber sequence in $\Ho(\Mod_{C})$
		\[
		 C \otimes^h_{B} \LL_{(N \to B)/(M \to A)} \to \LL_{(K \to C)/(M \to A)} \to \LL_{(K \to C)/(N \to B)}.
		\]
  \item	\emph{Base change.}
		Let
		\[
		 \begin{tikzcd}
		  (N' \to B')			& (N \to B)	\ar[l]	\\
		  (M' \to A')	\ar[u]	& (M \to A)	\ar[l]	\ar[u]
		 \end{tikzcd}
		\]
		be a homotopy pushout square in $s\Log\Ring^{\pre}$.
		Then there is an isomorphism in $\Ho(\Mod_{B'})$
		\[
		 B' \otimes^h_{B} \LL_{(N \to B)/(M \to A)} \cong \LL_{(N' \to B')/(M' \to A')}.
		\]
 \end{enumerate}
\end{proposition}

In order to apply these results in our setting of discrete prelog rings it would be useful to know when the homotopy pushouts appearing in (i) and (ii) coincide with the ordinary pushout.
The homotopy pushout in (i) appearing in the cofiber sequence is taken in the homotopy category of $\Mod_{C}$.
Suppose that $C$ and $B$ are discrete.
Then it is well known that
\[
 C \otimes_{B} \LL_{(N \to B)/(M \to A)} \cong C \otimes^h_{B} \LL_{(N \to B)/(M \to A)}
\]
in case $C$ is flat over $B$.
In the base change setting for discrete prelog rings it turned out to be easier to prove the base change result from scratch instead of deducing it from the homotopy version \cref{transitivity_BC} (ii) for simplicial prelog rings.

\begin{lemma} \label{BC_cotangent}
 Let
 \[
  \begin{tikzcd}
  (N' \to B')			& (N \to B)	\ar[l]	\\
  (M' \to A')	\ar[u]	& (M \to A)	\ar[l]	\ar[u]
 \end{tikzcd}
 \]
 be a pushout square in $\Log\Ring^{\pre}$ which is a homotopy pushout square in $s\Log\Ring^{\pre}$.
 Then
 \[
  \LL_{(N' \to B')/(M' \to A')} \cong \LL_{(N \to B)/(M \to A)} \otimes_A A'.
 \]
\end{lemma}

\begin{proof}
 Let $(K \to P) \to (N \to B)$ be a free simplicial resolution in the category of simplicial $(M \to A)$-algebras.
 Then the induced map
 \[
  (K \to P) \otimes_{(M \to A)} (M' \to A') \to (N \to B) \otimes_{(M \to A)} (M' \to A') = (N' \to B')
 \]
 represents the map from the homotopy pushout to the naive pushout, hence is a weak equivalence.
 It is therefore a free simplicial resolution of $(N' \to B')$ in the category of simplicial $(M' \to A')$-algebras and we can use it to compute the cotangent complex of $(N' \to B')$ over $(M' \to A')$:
 \begin{align*}
  \LL_{(N' \to B')/(M' \to A')} &= \Omega^1_{(K \to P) \otimes_{(M \to A)} (M' \to A')/(M' \to A')} \otimes_{P \otimes_A A'} (B \otimes_A A')	\\
								&= (\Omega^1_{(K \to P)/(M \to A)} \otimes_P B) \otimes_A A'	\\
								&= \LL_{(N \to B)/(M \to A)} \otimes_A A'.
 \end{align*}
\end{proof}

\begin{lemma} \label{homotopy_pushout_iff}
 Let
 \[
  \begin{tikzcd}
  (N' \to B')			& (N \to B)	\ar[l]	\\
  (M' \to A')	\ar[u]	& (M \to A)	\ar[l]	\ar[u]
 \end{tikzcd}
 \]
 be a pushout square in $\Log\Ring^{\pre}$.
 It is a homotopy pushout square if and only if the two pushout squares
 \begin{equation} \label{pushout_monoids_rings}
  \begin{tikzcd}
   N'			& N	\ar[l]			& B'			& B	\ar[l]	\\
   M'	\ar[u]	& M	\ar[l]	\ar[u]	& A'	\ar[u]	& A	\ar[l]	\ar[u]
  \end{tikzcd}
 \end{equation}
 are homotopy pushout squares in $s\Mon$ and $s\Ring$, respectively.
\end{lemma}

\begin{proof}
 Let $(N'' \to B'')$ represent the homotopy pushout of $(M \to A) \to (M' \to A')$ and $(M \to A) \to (N \to B)$.
 We obtain a map $(N'' \to B'') \to (N' \to B')$.
 By the definition of the model structure on $s\Log\Ring^{\pre}$ it is a weak equivalence if and only if $N'' \to N'$ and $B'' \to B'$ are weak equivalences.
 The pushout in the category of prelog rings is compatible with the pushouts in the category of monoids and the category of rings:
 \[
  B' \cong A' \otimes_A B	\qquad \text{and} \qquad N' \cong M' \sqcup_M N,
 \]
 i.e., the diagrams~(\ref{pushout_monoids_rings}) are pushout squares.
 Moreover, as $\Forget_{s\Mon \times s\Ring}^{s\Log\Ring^{\pre}}$ is a left Quillen functor, it preserves homotopy colimits.
 Therefore, $B''$ and $N''$ represent the homotopy pushouts of
 \[
  \begin{tikzcd}
	& N					& \text{and}	&		& B	\\
  M	& M	\ar[u]	\ar[l]	&				& A'	& A, \ar[u]	\ar[l]
  \end{tikzcd}
 \]
 in $s\Mon$ and~$s\Ring$, respectively.
 We conclude that $(N'' \to B'') \to (N' \to B')$ is a weak equivalence if and only if both $(N'' \to N')$ and $(B'' \to B')$ are.
\end{proof}

\begin{corollary} \label{condition_homotopy_pushout}
 Let
 \begin{equation} \label{log_pushout_square}
  \begin{tikzcd}
   (N' \to B')			& (N \to B)	\ar[l]	\\
   (M' \to A')	\ar[u]	& (M \to A)	\ar[l]	\ar[u]
  \end{tikzcd}
 \end{equation}
 be a pushout square in $\Log\Ring^{\pre}$.
 Assume that either of the ring homomorphisms $A \to B$ or $A \to A'$ is flat
  and that either $M \to N$ or $M \to M'$ is an integral homomorphism of integral monoids.
 Then the square (\ref{log_pushout_square}) is a homotopy pushout square.
\end{corollary}

\begin{proof}
 By \cref{homotopy_pushout_iff} we have to show that the two diagrams in (\ref{pushout_monoids_rings}) are homotopy pushout squares.
 For the diagram of rings this is well known.
 For the diagram of monoids this is \cite{Kato89}, Proposition~4.1.
\end{proof}

\begin{corollary} \label{localization_cotangent}
 Let
 \[
  (M \to A) \to (N \to B)
 \]
 be a homomorphism of prelog rings and $S \subseteq A$ a multiplicative subset.
 Then
 \[
  \LL_{(N \to S^{-1}B)/(M \to S^{-1}A)} \cong S^{-1}(\LL_{(N \to B)/(M \to A)}).
 \]
\end{corollary}

We finished our treatment of the compatibility of the logarithmic cotangent complex with base change.
The rest of this section uses transitivity and base change to compute the log cotangent complex for certain well behaved prelog rings.

\begin{proposition} \label{cotangent_polynomial}
 Let~$M \to A$ be a prelog ring and $N$ a finitely generated free monoid.
 Then
 \[
  H_i(\LL_{(M \oplus N \to A[N])/(M \to A)})
 \]
 vanishes for $i \ge 1$ and is isomorphic to $N^{\gp} \otimes A[N]$ for $i = 0$.
\end{proposition}

\begin{proof}
 By \cite{Ols05}, Theorem~8.16 we know that taking the associated log ring does not change the cotangent complex:
 \[
  \LL_{(N^a \to \ZZ[N])/(\{\pm 1\} \to \ZZ)} \cong \LL_{(N \to \ZZ[N])/(0 \to \ZZ)}.
 \]
 Since $(\{\pm 1\} \to \ZZ)$ is (obviously) log flat over $\ZZ$ with trivial log structure, Gabber's cotangent complex $\LL_{(N^a \to \ZZ[N])/(\{\pm 1\} \to \ZZ)}$ coincides with Olsson's (see \cite{Ols05}, Corollary 8.29), which we denote by $\LL^{\textit{Ols}}_{(N^a \to \ZZ[N])/(\{\pm 1\} \to \ZZ)}$.
 But
 \[
  \LL^{\textit{Ols}}_{(N^a \to \ZZ[N])/(\{\pm 1\} \to \ZZ)} \cong \Omega^1_{(N^a \to \ZZ[N])/(\{\pm 1\} \to \ZZ)}
 \]
 as $(\{\pm 1\} \to \ZZ) \to (N^a \to \ZZ[N])$ is log \replaced{smooth }{flat} (\cite{Ols05}, 1.1 (iii)) and
 \[
  \Omega^1_{(N^a \to \ZZ[N])/(\{\pm 1\} \to \ZZ)} \cong \Hom_\Mon(N,\ZZ[N]).
 \]
 Now consider the pushout square
 \[
  \begin{tikzcd}
   (M \oplus N \to A[N]			& (N \to \ZZ[N])	\ar[l]			\\
   (M \to A)			\ar[u]	& (\{1\} \to \ZZ).		\ar[l]	\ar[u]
  \end{tikzcd}
 \]
 The ring homomorphism $\ZZ \to \ZZ[N]$ is flat and the monoid $N$ is integral.
 Hence, by \cref{condition_homotopy_pushout}, the above square is a homotopy pushout square.
 Applying \cref{BC_cotangent} yields an isomorphism
 \[
  \LL_{(M \oplus N \to A[N])/(M \to A)} \cong \LL_{(N \to \ZZ[N])/(\{1\} \to \ZZ)} \otimes_\ZZ A.
 \]
 From this and the above description of $\LL_{(N \to \ZZ[N])/(\{1\} \to \ZZ)}$ we obtain the result.
\end{proof}

\begin{proposition} \label{cotangent_regular_ideal}
 In the situation of \cref{cotangent_polynomial} let~$I$ be a regular ideal of $A[N]$.
 Then
 \[
  \LL_{(M \oplus N \to A[N]/I)/(M \to A)} \cong (I/I^2 \overset{-d}{\longrightarrow} \Omega^1_{(M \oplus N \to A[N])/(M \to A)} \otimes_{A[N]} A[N]/I),
 \]
 where $I/I^2$ is placed in degree $-1$ and~$d$ is induced from the differential.
\end{proposition}

\begin{proof}
 The proof is the same as for Olsson's cotangent complex (\cite{Ols05}, Lemma~6.9):
 By \cref{transitivity_BC} (i) we have a homotopy cofiber sequence
 \[
  \LL_{(M \oplus N \to A[N])/(M \to A)} \otimes^h_{A[N]} A[N]/I \to \LL_{(M \oplus N \to A[N]/I)/(M \to A)} \to \LL_{(M \oplus N \to A[N]/I)/(M \oplus N \to A[N])}.
 \]
 \cref{cotangent_polynomial} gives us
 \[
  \LL_{(M \oplus N \to A[N])/(M \to A)} \otimes^h_{A[N]} A[N]/I \cong \Omega^1_{(M \oplus N \to A[N])/(M \to A)} \otimes_{A[N]} A[N]/I.
 \]
 Moreover,
 \[
  \LL_{(M \oplus N \to A[N]/I)/(M \oplus N \to A[N])} \cong \LL_{(A[N]/I)/A[N]} \cong I/I^2[1].
 \]
 We have the first isomorphism because the map on monoids is the identity (\cite{Ols05}, Lemma~8.17) and the second one is a classical result for the cotangent complex of rings (\cite{Ill71}, III, Proposition~3.2.4).

 It remains to show that the resulting map $I/I^2 \to \Omega^1_{(M \oplus N \to A[N])/(M \to A)}$ is given by the negative of the differential.
 By functoriality we have a factorization
 \[
  I/I^2 \to \Omega^1_{A[N]/A} \to \Omega^1_{(M \oplus N \to A[N])/(M \to A)}.
 \]
 The first map is the negative of the differential by \cite{Ill71}, III Proposition~1.2.9 and the second map is the canonical one.
\end{proof}

\begin{corollary} \label{cotangent_free_monoid}
 Let~$A$ be a ring and~$M$ a finitely generated free commutative submonoid of~$A^{\times}$.
 Then $\LL_{(M \to A)/(\{1\} \to A)}$ is concentrated in degree zero.
\end{corollary}

\begin{proof}
 We choose generators $m_1,\ldots,m_r$ of~$M$.
 This defines an isomorphism of $A[M]$ with $A[T_1,\ldots,T_r]$.
 Let~$I$ be the ideal of $A[T_1,\ldots,T_r]$ generated by $T_i - m_i$ for $i = 1,\ldots,r$.
 This is clearly a regular ideal.
 By \cref{cotangent_regular_ideal} we have
 \[
  \LL_{(M \to A)/(\{1\} \to A)} \cong (I/I^2 \overset{-d}{\longrightarrow} \Omega^1_{(M \to A[M])/(\{1\} \to A)} \otimes_{A[M]} A).
 \]
 We have a natural identification of $I/I^2$ with the free $A$-module with generators $(T_i-m_i)$.
 Moreover, by \cref{cotangent_polynomial}, $\Omega^1_{(M \to A[M])/(\{1\} \to A)} \otimes_{A[M]} A)$ is isomorphic to $M^{\gp} \otimes A$.
 The differential~$d$ maps $(T_i -m_i)$ to $dT_i = T_i (dT_i/T_i)$ (corresponding to $m_i \otimes m_i \in M^{\gp} \otimes A$).
 This map is injective.
\end{proof}

Finally, we will need that the cotangent complex is compatible with filtered colimits:

\begin{proposition} \label{colimits_cotangent}
 Let $(M \to A) = \colim_{i \in I} (M_i \to A_i)$ and $(N \to B) = \colim_{i \in I} (N_i \to B_i)$  be filtered colimits in the category of prelog rings.
 Suppose we are given compatible homomorphisms $(M_i \to A_i) \to (N_i \to B_i)$.
 Then there is a natural isomorphism
 \[
  \LL_{(N \to B)/(M \to A)} \cong \colim_{i \in I} \LL_{(N_i \to B_i)/(M_i \to A_i)}.
 \]
\end{proposition}

\begin{proof}
 The functors~$F$ and~$G$ in the canonical resolution~(\ref{canonical_resolution}) commute with filtered colimits and so does the formation of log differentials.
\end{proof}

\section{Unramified and tame extensions}	\label{section_unramified_tame}

For a valued field~$K$ we will adopt the following notation.
The valuation of an element~$x$ in~$K$ is written $|x|_K$ or only $|x|$ when it does not cause confusion.
We denote the valuation ring of~$K$ by~$K^+$ and the value group of the valuation by $\Gamma_K$.
We endow~$K^+$ with the \emph{total log structure} $(K^+ \setminus \{0\} \to K^+)$.
For an extension $L|K$ of valued fields we define
\[
 \LL^{\log}_{L/K} := \LL_{(L^+ \setminus \{0\} \to L^+)/(K^+ \setminus \{0\} \to K^+)}.
\]

Remember that a finite extension $L|K$ of valued fields is \emph{unramified} if $L^{\sh} = K^{\sh}$ (strict henselization).
It is \emph{tamely ramified} (or tame for short)  if $[L^{\sh}:K^{\sh}]$ is prime to the residue characteristic of~$K^+$.
In this case $[\Gamma_L:\Gamma_K] = [L^{\sh}:K^{\sh}]$.
An algebraic extension $L|K$ of valued fields is tame if all its finite subextensions are tame.

\begin{lemma} \label{cotangent_unramified}
 Let $L|K$ be unramified.
 Then $\LL^{\log}_{L/K} \cong 0$.
 In particular, $\Omega^{\log}_{L/K} = 0$.
\end{lemma}

\begin{proof}
 Since $L|K$ is unramified, $\Gamma_L = \Gamma_K$, so the total log structure of $L^+$ is the logification of $(K^+ \setminus \{0\} \to L^+)$.
 We can thus compute the logarithmic cotangent complex as follows:
 \[
  \LL^{\log}_{L/K} \cong \LL_{(K^+ \setminus \{0\} \to L^+)/(K^+ \setminus \{0\} \to K^+)} \cong \LL_{L^+/K^+} \cong 0.
 \]
 The left hand isomorphism is due to \cite{Ols05}, Theorem~8.16, the middle one to \cite{Ols05}, Lemma~8.17, and the right hand one to \cite{GR03}, Theorem~6.3.32 and the well known fact that the differentials vanish for unramified extensions.
\end{proof}

\begin{proposition} \label{cotangent_tame}
 For any tame extension $L|K$ of valued fields the logarithmic cotangent complex is trivial: $\LL^{\log}_{L/K} \cong 0$.
 In particular, $\Omega^{\log}_{L/K} \cong 0$.
\end{proposition}

\begin{proof}
 Using \cref{cotangent_unramified} and transitivity (\cref{transitivity_BC}~(i)) for the extensions in the diagram
  \[
  \begin{tikzcd}
							& L^{\sh}	\ar[dl,dash] \ar[dr,dash]	\\
   K^{\sh}	\ar[dr,dash]	&										& L			\ar[dl,dash]				\\
							& K
  \end{tikzcd}
 \]
 we reduce to the case where~$K$ is strictly henselian.
 Moreover, since the logarithmic cotangent complex is compatible with filtered colimits (\cref{colimits_cotangent}), we can reduce to the case of a finite extension.
 Remember that the Galois group of a finite tame extension of a strictly henselian valued field is abelian, so in particular solvable.
 Hence, we can decompose the extension $L|K$ into a chain of subextensions of prime degree:
 \[
  K = L_0 \subseteq L_1 \subseteq \ldots \subseteq L_n = L
 \]
 such that $[L_{i+1}:L_i]$ is a prime number.
 Transitivity (\cref{transitivity_BC}~(i)) allows us to treat each extension separately.
 We may thus assume $[L:K]$ is a prime number~$\ell$ (prime to the residue characteristic as $L|K$ is tame).

 We have $L = K[a^{1/\ell}]$ for some $a \in K$ with $|a| < 1$.
 The valuation ring~$L^+$ is the filtered colimit of its subalgebras $R_b = K^+[ba^{1/\ell}]$ with $b \in K$ such that $|ba^{1/\ell}| < 1$ (see the proof of \cite{GR03}, Proposition.3.13~(i)).
 We equip~$R_b$ with the prelog structure
 \[
  M_b := [(K^+ \setminus \{0\} \oplus \N)/\sim] \longrightarrow R_b,
 \]
 where the equivalence relation is generated by $(b^{\ell}a,0) \sim (1,\ell)$ (note that the first component is written multiplicatively and the second one additively) and $(x,r) \in M_b$ is mapped to $x(ba^{1/\ell})^r \in R_b$.
 We claim that the total log structure of~$L^+$ is the logification of the colimit of the prelog rings $(M_b \to R_b)$.
 Since~$M_b$ and~$R_b$ are naturally contained in~$L^+$ and we already know that $L^+ = \colim_b R_b$, this amounts to checking that every element $y \in L^+ \setminus \{0\}$ can be written in the form $y = ux(ba^{1/\ell})^r$ for $x \in K^+$, $r \in \N$, $b \in K$ such that $|ba^{1/\ell}|_L < 1$, and~$u$ a unit of~$L^+$.
 We choose $b \in K$ such that $y \in R_b$.
 Then $|y|_L = |x(ba^{1/\ell})^r|_L$ for some~$x$ and~$r$ as above.
 Setting $u = y x^{-1}(ba^{1/\ell})^{-r}$, the claim follows.

 Using that logification does not change the cotangent complex (\cite{Ols05}, Theorem~8.16) and \cref{colimits_cotangent} this reduces us to showing that $\LL_{(M_b \to R_b)/(K^+ \setminus \{0\} \to K^+)}$ is concentrated in degree~$0$.

 We now consider the following pushout square of prelog rings:
 \[
  \begin{tikzcd}
   (M_b \to R_b)							& (\N \to R_b)	\ar[l]			\\
   (K^+ \setminus \{0\} \to K^+)	\ar[u]	& (\N \to K^+)	\ar[l]	\ar[u],
  \end{tikzcd}
 \]
 where the prelog structures on the right hand side are given by $r \mapsto (ba^{1/\ell})^r$ and $r \mapsto (b^{\ell}a)^r$, the right hand vertical monoid homomorphism is $r \mapsto \ell r$, the upper horizontal one $r \mapsto (1,r)$, and the lower one $r \mapsto (b^{\ell}a)^r$.
 The identity on~$K^+$ is (obviously) flat and the right hand vertical map of monoids is integral.
 Hence, the diagram is also a homotopy pushout (see \cref{condition_homotopy_pushout}).
 We conclude that
 \[
  \LL_{(M_b \to R_b)/(K^+ \setminus \{0\} \to K^+)} \cong \LL_{(\N \to R_b)/(\N \to K^+)}.
 \]
 By \cite{Ols05}, Theorem~8.16
 \[
  \LL_{(\N \to K^+)/(\N \to R_b)} \cong \LL_{(\N \to K^+)^a/(\N \to R_b)^a}.
 \]
 The logification of $(\N \to K^+) \to (\N \to R_b)$ is a log \'etale, integral homomorphism of fine, integral log rings.
 Its cotangent complex is thus isomorphic to Olsson's cotangent complex (\cite{Ols05}, Corollary~8.29), which in turn is concentrated in degree zero by log smoothness (\cite{Ols05}, (1.1 (iii))).
 Moreover, it vanishes in degree zero by \cite{Ogus18}, Chapter~IV, Proposition~3.1.3.
\end{proof}

We would like to extend our results to tame extensions of local Huber pairs.
Recall from \cite{HueAd}, Definition~6.1, that a Huber pair $(A,A^+)$ is local if~$A$ is a local ring with maximal ideal~$\m_A$ and~$A^+$ is the preimage in~$A$ of a valuation ring of~$A/\m_A$.
Moreover, remember that a local map
\[
 (A,A^+) \longrightarrow (B,B^+)
\]
of local Huber pairs is \emph{tame} if $A \to B$ is essentially étale (i.e., $B$ is a localization of an étale $A$-algebra) and the extension of residue fields
\[
 (k_A,k_A^+) \longrightarrow (k_B,k_B^+)
\]
is tame.

Given a Huber pair $(A,A^+)$, we endow~$A^+$ with the log structure
\[
 A^+ \cap A^\times \longrightarrow A^+
\]
If $\Spec A \to \Spec A^+$ is an open immersion, it corresponds to the compactifying log structure on $\Spec A^+$ defined by the open subset $\Spec A$.
In case $(A,A^+)$ is local, this log structure can also be described as
\[
 A^+ \setminus \m_A \longrightarrow A^+
\]
and we call it the \emph{total log structure} on~$A^+$.
For a valued field $(k,k^+)$ we recover the total log structure on~$k^+$ we considered earlier.

\begin{corollary} \label{tame_local_Huber_pair_cotangent}
 Let
 \[
  (A,A^+) \longrightarrow (B,B^+)
 \]
 be a tame extension of noetherian local Huber pairs.
 Then
 \[
  \LL^{\log}_{(B,B^+)/(A,A^+)} \cong 0.
 \]
 In particular, $\Omega^{\log}_{(B,B^+)/(A,A^+)} \cong 0$.
\end{corollary}

\begin{proof}
 We consider the short exact sequence
 \[
  0 \longrightarrow \m_B \longrightarrow B^+ \longrightarrow k_B^+ \longrightarrow 0.
 \]
 of~$B^+$-modules.
 Thensoring with~$\LL^{\log}_{(B,B^+)/(A,A^+)}$ we obtain an exact triangle
 \begin{equation} \label{exact_triangle}
  \LL^{\log}_{(B,B^+)/(A,A^+)} \otimes^h_{B^+} \m_B \longrightarrow \LL^{\log}_{(B,B^+)/(A,A^+)} \longrightarrow \LL^{\log}_{(B,B^+)/(A,A^+)} \otimes^h_{B^+} k_B^+ \overset{+1}{\longrightarrow} \ldots
 \end{equation}
 Since~$\m_B$ is a $B$-module, we can rewrite the leftmost part as
 \[
  \LL^{\log}_{(B,B^+)/(A,A^+)} \otimes^h_{B^+} \m_B \cong \LL^{\log}_{(B,B^+)/(A,A^+)} \otimes^h_{B^+} B \otimes^h_B \m_B.
 \]
 We know that~$A$ is a localization of~$A^+$ by the multiplicative subset $S:= A^+ \setminus \m_A$.
 Then
 \[
  (A,A) = (A,S^{-1} A^+) \longrightarrow (B,S^{-1} B^+)
 \]
 is a tame local homomorphism of local Huber pairs.
 The corresponding extension of residue fields should be a finite extension of the trivially valued field~$k_A$.
 Therefore $S^{-1} B^+/\m_B$ is a field and this already implies $B = S^{-1} B^+$.
 Now we can use \cref{localization_cotangent} to compute
 \[
  \LL^{\log}_{(B,B^+)/(A,A^+)} \otimes^h_{B^+} B = S^{-1} \LL^{\log}_{(B,B^+)/(A,A^+)} \cong \LL^{\log}_{(B^+ \cap B^\times \to B)/(A^+ \cap A^\times \to A)}.
 \]
 The logification of $(B^+ \cap B^\times \to B)$ is the trivial log structure and similarly for~$A$.
 Hence, \cite{Ols05}, Theorem~8.16 and Lemma~8.17 imply
 \[
  \LL^{\log}_{(B^+ \cap B^\times \to B)/(A^+ \cap A^\times \to A)} \cong \LL_{B/A},
 \]
 which is trivial as $B/A$ is essentially étale.
 Putting the pieces together we obtain
 \[
  \LL^{\log}_{(B,B^+)/(A,A^+)} \otimes^h_{B^+} \m_B \cong 0.
 \]
 In order to compute
 \[
  \LL^{\log}_{(B,B^+)/(A,A^+)} \otimes^h_{B^+} k_B^+,
 \]
 we need to check that the diagram
 \[
  \begin{tikzcd}
   (k_B^+ \cap k_B^\times \to k_B^+)			& (B^+ \cap B^\times \to B^+)	\ar[l]	\\
   (k_A^+ \cap k_A^\times \to k_A^+)	\ar[u]	& (A^+ \cap A^\times \to A^+)	\ar[l]	\ar[u]
  \end{tikzcd}
 \]
 is a homotopy pushout square.
 The two diagrams
 \[
  \begin{tikzcd}
   k_B^+ \cap k_B^\times 			& B^+ \cap B^\times	\ar[l]			& k_B^+			& B^+	\ar[l]			\\
   k_A^+ \cap k_A^\times	\ar[u]	& A^+ \cap A^\times	\ar[l]	\ar[u]	& k_A^+	\ar[u]	& A^+	\ar[l]	\ar[u]
  \end{tikzcd}
 \]
 are pushout diagrams as $A \to B$ is unramified, whence $\m_A B^+ = \m_A B = \m_B$.
 We want to check that they are also homotopy pushout diagrams.
 For the left hand diagram this is the case because $A^+ \cap A^\times \to k_A^+ \cap k_A^\times$ is a quotient by the submonoid $1+\m_A \subseteq A^+ \cap A^\times$, which implies that it is an integral morphism (see \cite{KatoIntegralLog}, Lemma~2.2~(3)).
 The right hand diagram is a homotopy pushout as $A \to B$ is flat, which implies that $A^+ \to B^+$ is flat (see \cite{HueAd}, Proposition~11.8).
 \cref{condition_homotopy_pushout} implies that the above diagram of prelog rings is a homotopy pushout diagram.
 We can now use \cref{transitivity_BC} to compute the base change of $\LL^{\log}_{(B,B^+)/(A,A^+)}$ to $k_B^+$:
 \[
  \LL^{\log}_{(B,B^+)/(A,A^+)} \otimes^h_{B^+} k_B^+ \cong \LL^{\log}_{(k_B,k_B^+)/(k_A,k_A^+)}.
 \]
 We have seen in \cref{cotangent_tame} that the logarithmic cotangent complex of $k_B/k_A$ is trivial.
 Going back to the exact triangle (\ref{exact_triangle}) we see that two out of three terms vanish, implying that $\LL^{\log}_{(B,B^+)/(A,A^+)} \cong 0$.
\end{proof}

\section{Logarithmic differentials on adic spaces} \label{section_log_diff_adic}

All Huber pairs in this section will be endowed with the discrete topology and all adic spaces will be discretely ringed, i.e., locally isomorphic to a Huber pair with the discrete topology.
For a morphism of local Huber pairs $(A,A^+) \to (B,B^+)$, we define
\[
 \Omega^{n,\log}_{(B,B^+)/(A,A^+)} := \Omega^n_{(B^+ \cap B^{\times} \to B^+)/(A^+ \cap A^{\times} \to A^+)}.
\]
If $n = 1$, we omit~$n$ and just write $\Omega^{\log}_{(B,B^+)/(A,A^+)}$.
For this section we fix a field~$k$ and a valuation ring~$k^+$ of~$k$.
We assume that one of the following properties is satisfied:
\begin{itemize}
 \item the residue characteristic of~$k^+$ is~$0$,
 \item $k$ is algebraically closed,
 \item $k = k^+$ is perfect.
\end{itemize}

For a Huber pair $(A,A^+)$ over $(k,k^+)$ we use the short notation $\Omega^n_{A^+}$ for $\Omega^n_{A^+/k^+}$ and $\Omega^{n,\log}_{(A,A^+)}$ for $\Omega^{n,\log}_{(A,A^+)/(k,k^+)}$.

\subsection{Logarithmic differentials on local Huber pairs}

The following is a reformulation of results of \cite{GR03}, \S~6.5.

\begin{proposition} \label{Omega_valuation_tf}
 Let $(K,K^+)$ be any extension of valued fields of $(k,k^+)$.
 Then $\Omega^{n,\log}_{(K,K^+)}$ and $\Omega^n_{K^+}$ are torsion free for all $n \ge 1$.
\end{proposition}

\begin{proof}
 The statement about $\Omega_{K^+}$ is \cite{GR03}, Theorem~6.5.15 and Corollary~6.5.21.
 In \cite{GR03}, \S~6.5 Gabber and Ramero examine the natural homomorphism
 \[
  \rho_{K^+/k^+} : \Omega_{k^+/\ZZ}^{\log} \otimes_{k^+} K^+ \to \Omega_{K^+/\ZZ}^{\log}.
 \]
 The cokernel of $\rho_{K^+/k^+}$ is isomorphic to $\Omega^{\log}_{(K,K^+)}$ (\cite{Ogus18}, Chapter~IV, Proposition~2.3.1).
 Therefore the result for $n = 1$ follows from \cite{GR03}, Lemma~6.5.16 if the residue characteristic of~$k^+$ is~$0$, from Theorem~6.5.20 if~$k$ is algebraically closed, and from Corollary~6.5.21 if $k=k^+$ is perfect.
 The general case ($n >1$) follows as well as over a valuation ring exterior products of torsion free modules are torsion free (\cite{HKK17}).
\end{proof}

\begin{corollary} \label{Omega_injective}
 Let $(A,A^+)$ be a local Huber pair over~$(k,k^+)$.
 Then the natural homomorphism
 \[
  \varphi: \Omega^{n,\log}_{(A,A^+)} \longrightarrow \Omega^{n,\log}_{(A,A^+)} \otimes_{A^+} A \cong \Omega^n_{A/k}
 \]
 is injective.
 In particular, $\Omega^{\log}_{(A,A^+)}$ is torsion free if $\Omega^n_{A/k}$ is torsion free.
\end{corollary}

\begin{proof}
 We denote by~$\m$ the maximal ideal of~$A$ and set $(K,K^+) = (A/\m,A^+/\m)$.
 Note first that $A$ is the localization of~$A^+$ by the multiplicative subset $S = A^+ \setminus \m$.
 In order to show that $\varphi$ is injective, we need to show that $\Omega^{n,\log}_{(A,A^+)}$ is $S$-torsion free.
 Let $s \in S$ and $\omega \in \Omega^{n,\log}_{(A,A^+)}$ such that $s\omega = 0$.
 By \cref{Omega_valuation_tf} we know that $\Omega^{n,\log}_{(A,A^+)} \otimes_A^+ A^+/\m \cong \Omega^{n,\log}_{(K,K^+)}$ is torsion free.
 We conclude that $\omega \in \m \Omega^{n,\log}_{(A,A^+)}$.
 Since~$\m$ is an ideal of~$A$, the action of~$A^+$ on $\m \Omega^{n,\log}_{(A,A^+)}$ extends to~$A$.
 But~$s$ is a unit in~$A$, so $s\omega = 0$ implies $\omega = 0$.
\end{proof}

If $(A,A^+)$ is the localization of a smooth Huber pair over $(k,k^+)$, $\Omega^n_{A/k}$ is finitely generated and free.
We do not expect an analogous statement for logarithmic differentials.
However, we have the following result.
%
%
%
%
%

\begin{proposition} \label{local_torsion_free}
 Let $(A,A^+)$ be a local $(k,k^+)$-algebra such that $A$ is the localization of a smooth $k$-algebra.
 Then $\Omega^n_{A^+}$ and $\Omega_{(A,A^+)}^{n,\log}$ are flat $A^+$-modules.
\end{proposition}

\begin{proof}
%
 Let~$\m$ be the maximal ideal of~$A$ and set $(K,K^+) = (A/\m,A^+/\m)$.
 Since~$A$ is the localization of a smooth $k$-algebra,~$\Omega^n_{A/k}$ is flat (even free).
 Therefore, $\Omega^{n,\log}_{(K,K^+)}$ and $\Omega^n_{K^+/k^+}$ are torsion free by \cref{Omega_valuation_tf}.
 We have the following identifications:
 \[
  \Omega^{n,\log}_{(A,A^+)} \otimes_{A^+} A \cong \Omega^n_{A/k},	\qquad \Omega^{n,\log}_{(A,A^+)} \otimes_{A^+} K^+ \cong \Omega^{n,\log}_{(K,K^+)},
 \]
 and similarly for $\Omega^n_{A^+/k^+}$.
 The result thus follows from the flatness criterion \cite{HueAd}, Proposition~11.8.
\end{proof}

\subsection{The presheaf of logarithmic differentials} \label{section_presheaf_log_diff}

The naive idea of defining logarithmic differentials on an adic space $\cX$ is to set for an affinoid open $\Spa(A,A^+)$
\[
 \Omega^{\log}(\Spa(A,A^+)) = \Omega^{\log}_{(A,A^+)}
\]
and to glue these for general open subspaces.
This approach is too naive for various reasons.
Unfortunately the sheaf condition is not satisfied.
Consider for instance the following

\begin{example}
 Let $\cX$ be the affinoid adic space $\Spa(k[T,T^{-1}],k)$ over an algebraically closed field~$k$.
 On the one hand,
 \[
  \Omega^{\log}_{(k[T,T^{-1}],k)} = \Omega_{(k \setminus \{0\} \to k)/(k \setminus \{0\} \to k)} = 0.
 \]
 On the other hand, $\cX$ is covered by the affinoid open subspaces $\Spa(k[T,T^{-1}],k[T])$ and $\Spa(k[T,T^{-1}],k[T^{-1}])$.
 The logarithmic differentials $dT/T$ and $-dT^{-1}/dT^{-1}$ on $\Spa(k[T,T^{-1}],k[T])$ and $\Spa(k[T,T^{-1}],k[T^{-1}])$, respectively, coincide on the intersection but do not lift to a global section.
 Hence, the sheaf condition is not satisfied.

 Apart from the fact that the above defined presheaf of logarithmic differentials is not a sheaf, its sections on $\Spa(k[T,T^{-1}],k)$  are not the ones we would expect.
 Intuitively there should be a global section lifting $dT/T$ and $-dT^{-1}/dT^{-1}$.
\end{example}

To overcome the problem described in the example we only work with strict affinoids, which are defined as follows.

\begin{definition}
 We say that a Huber pair $(A,A^+)$ is \emph{strict} if~$A$ is a localization of~$A^+$.
 An affinoid adic space $\Spa(A,A^+)$ is \emph{strict} if $(A,A^+)$ is strict.
 For an adic space~$\cX$ we denote the category of strict affinoid open subspaces by $\cX_{\straff}$.
\end{definition}

\begin{lemma} \label{strict_neighborhood_basis}
 Let $\cX$ be an adic space locally of the form $\Spa(A,A^+)$ with $A/A^+$ essentially of finite type (i.e., $A$ is a localization of an $A^+$-algebra of finite type).
 Then the strict affinoids of~$\cX$ form a basis of the topology.
\end{lemma}

\begin{proof}
 Without loss of generality we may assume that $\cX$ is of the form $\Spa(A,A^+)$ with $A/A^+$ essentially of finite type.
 Given an affinoid open subspace $\Spa(B,B^+)$ we have a diagram
 \[
  \begin{tikzcd}
    B			& A		\ar[l]			\\
    B^+	\ar[u]	& A^+	\ar[l]	\ar[u]
  \end{tikzcd}
 \]
 such that $\Spec B \to \Spec A$ is an open immersion and $B^+$ is the normalization in $B$ of an $A^+$-algebra of finite type.
 In particular,~$B$ is essentially of finite type over~$B^+$.
 It thus has a compactification $Y \to \Spec B^+$.
 By \cite{HueAd}, Lemma~7.5 we have an identification $\Spa(B,B^+) = \Spa(B,Y)$.
 Covering~$Y$ by affines $\Spec A_i^+$ and each $\Spec B \cap \Spec A_i^+$ by affines $\Spec A_{ij}$, we obtain a cover of $\Spa(B,Y)$ by the strict affinoids $\Spa(A_{ij},A_i^+)$.
 \end{proof}

 \begin{lemma} \label{tensor_product}
 Let $(A,A^+)$, $(B,B^+)$, and $(C,C^+)$ be strict affinoids.
 Then the tensor product
 \[
  (D,D^+) = (B,B^+) \otimes_{(A,A^+)} (C,C^+)
 \]
 is strict.
\end{lemma}

\begin{proof}
 Let $S = D^+ \cap D^{\times}$.
 We claim that $D = S^{-1}D^+$.
 Every element of~$S$ is invertible in~$D$, whence the existence of a natural homomorphism $S^{-1}D^+ \to D$.
 Injectivity is clear as $D^+ \subset D$.
 Let $d \in D$.
 We want to write $d = d^+/s$ for $d^+ \in D^+$ and $s \in S$.
 Without loss of generality we may assume $d = b \otimes c$ for $b \in B$ and $c \in C$.
 But by assumption $b = b^+/s$ and $c = c^+/t$ for $b^+ \in B^+$, $s \in B^+ \cap B^{\times}$, and $b \in C^+ \cap C^{\times}$.
 Hence $d = (b^+ \otimes c^+)/(s \otimes t)$.
\end{proof}

Let $X \to S$ be a morphism of schemes which is essentially of finite type.
We equip $\Spa(X,S)_{\straff}$ (see \cite{Tem11}, \S~3.1 for the definition of $\Spa(X,S)$ for a morphism of schemes $X \to S$) with the topology whose coverings are surjective families.
Note that by \cref{tensor_product} the necessary fiber products for the structure of a site exist.
We denote by $\Spa(X,S)_{\Top}$ the site associated with the topological space $\Spa(X,S)$.
By \cref{strict_neighborhood_basis} the corresponding topoi of $\Spa(X,S)_{\straff}$ and $\Spa(X,S)_{\Top}$ are equivalent.
If~$\mathcal{F}$ is a presheaf on $\Spa(X,S)_{\straff}$ we can view its sheafification as a sheaf~$\mathcal{G}$ on all of $\Spa(X,S)$.
Slightly abusing notation we will say that~$\mathcal{G}$ is the sheafification of~$\mathcal{F}$.
We have thus justified the restriction to strict affinoids.

Our presheaf of interest is the presheaf of logarithmic differentials~$\Omega^{\log}$.
It is defined on $\Spa(X,S)_{\straff}$ as
\[
 \Omega^{\log}(\Spa(A,A^+)) := \Omega^{\log}_{(A,A^+)}.
\]
Similarly we define~$\Omega^{n,\log}$ by
\[
 \Omega^{n,\log}(\Spa(A,A^+)) := \Omega^{n\log}_{(A,A^+)}.
\]
Even restricted to strict affinoids $\Omega^{\log}$ is not a sheaf as the following example shows.

\begin{example}
 For positive integers $d$ and $r$ we consider the action of $\mu_d$ on $\bC[T_0,\ldots,T_r]$ induced by the diagonal embedding of~$\mu_d$ in $\Gl_{r+1}(\bC)$.
 In other words, $\xi \in \mu_d$ acts by multiplying each coordinate with~$\xi$.
 We consider the quotient spaces
 \[
  X_{r,d} := (\Spec \bC[T_0,\ldots,T_r])/\mu_d = \Spec \bC[T_0,\ldots,T_r]^{\mu_d}.
 \]
 They are normal and can also be described as the affine cone of the $d$-th Veronese embedding of~$\P_C^r$.
 Moreover, note that
 \[
  A^+_{r,d} := \bC[T_0,\ldots,T_r]^{\mu_d}
 \]
 is the $\bC$-subalgebra of $\bC[T_0,\ldots,T_r]$ generated by all monomials of degree~$d$.
 In \cite{GrRo11}, Proposition~4.1 it is shown that $\Omega_{X_{r,d}}$ has torsion if and only if $d \ge 3$.

 Let $U_{r,d} = \Spec A_{r,d}$ be the open subscheme of~$X_{r,d}$ defined by inverting $T_0^d$.
 Then $(A_{r,d},A_{r,d}^+)$ is a strict Huber pair.
 By transitivity (\cref{transitivity_BC} (i)) and the vanishing of $H_1(\LL_{((T_0^d)^{\N} \to A_{r,d}^+)/(\{0\} \to A_{r,d}^+)})$ (\cref{cotangent_free_monoid}) we know that
 \[
  \Omega_{A_{r,d}^+/k} \to \Omega^{\log}_{(A_{r,d},A_{r,d}^+)/(\bC,\bC)}
 \]
 is injective.
 Hence, $\Omega^{\log}_{(A_{r,d},A_{r,d}^+)/(\bC,\bC)}$ has torsion as well for $d \ge 3$.

 Let $Y_{r,d} \to X_{r,d}$ be the blowup in the origin.
 Denote by~$D$ the Cartier divisor of~$Y_{r,d}$ which is the pullback of the Cartier divisor of~$X_{r,d}$ defined by $X_0^d$.
 Then~$Y_{r,d}$ is smooth and~$D$ is a simple normal crossings divisor.
 In particular, $(U_{r,d},Y_{r,d})$ is log smooth, so $\Omega^{\log}_{(U_{r,d},Y_{r,d})/(\bC,\bC)}$ is torsion free.

 We cover~$Y_{r,d}$ by affine schemes $\Spec B_i^+$.
 As the complement of~$U_{r,d}$ in~$Y_{r,d}$ is the support of a principal Cartier divisor, the intersection of $\Spec B_i$ with $U_{r,d}$ is still affine.
 We denote the corresponding ring by~$B_i$.
 The strict affinoids $\Spa(B_i,B_i^+)$ cover $\Spa(A_{r,d},A_{r,d}^+)$.
 Moreover, $\Omega^{\log}_{(B_i,B_i^+)/(\bC,\bC)}$ is a finitely generated free $B_i^+$-module as $(B_i,B_i^+)$ is log smooth over $(\bC,\bC)$.
 In particular, $\Omega^{\log}_{(B_i,B_i^+)/(\bC,\bC)}$ is torsion free over~$A_{r,d}^+$.
 But then
 \[
  \Omega^{\log}_{(A_{r,d},A_{r,d}^+)/(\bC,\bC)} \to \prod_i \Omega^{\log}_{(B_i,B_i^+)/(\bC,\bC)}
 \]
 cannot be injective because $\Omega^{\log}_{(A_{r,d},A_{r,d}^+)/(\bC,\bC)}$ has torsion.
 We conclude that the sheaf axiom is not satisfied.
\end{example}

The example already suggests that the problems lie in the singularities of $\Spec A^+$.
Indeed we will see in \cref{section_differentials_smooth} that the differentials are well behaved for log smooth Huber pairs.

We are now interested in the sheafification of~$\Omega^{n,\log}$.
For a strict Huber pair $(A,A^+)$ over $(k,k^+)$ and $n \ge 1$, we consider the natural map
\[
 \Omega^{n,\log}_{(A,A^+)}\to \Omega^{n,\log}_{(A,A^+)} \otimes_{A^+} A \overset{\sim}{\to} \Omega_A.
\]
If $(A,A^+)$ is local, it is injective by \cref{Omega_injective}.
We conclude that the sheafification of $\Omega^{n,\log}$ is a subsheaf of $\Omega^n$.
It will turn out that the sheafification can be described in terms of the K\"ahler seminorm which we study in \cref{section_Kaehler}.

\section{Rangers} \label{section_rangers}

In \cite{Tem16}, \S~5, Temkin studies the Kähler seminorm on the differentials $\Omega_{K/k}$ for an extension of real valued fields $K/k$.
It is defined as
\[
 |\omega|_{\Omega} := \inf_{\omega = \sum f_i dg_i} \max_i \{|f_i|_A |g_i|_A\},
\]
where the infimum is taken over all representations of~$\omega$ as a finite sum $\omega = \sum_i f_i dg_i$.
Note that this infimum does not necessarily exist as an element of the value group of~$K$.
It is rather a nonnegative real number once we have fixed an embedding of the value group into $\R_{>0}$.

We aim for an analogous definition of the Kähler seminorm for extensions of local Huber pairs.
However, we need to investigate where to take the infimum.
In general, the value group~$\Gamma$ of the valuation on the local Huber pair has higher rank and does not embed into the real numbers.
This section presents a construction that can be seen as a completion of~$\Gamma$, called the set of \emph{rangers}.
\added{In joint work with Michael Temkin we elaborate this construction in more detail, see \S 2 in \cite{adiccurves}.}
What we are using here is self contained.
We also want to remark that similar constructions are discussed in \cite{KuhNartCuts}.
However, their focus lies more in the classification of ordered groups.

\begin{definition}
 For an abelian ordered group~$\Gamma$ (written multiplicatively) a \emph{ranger} is a pair $(\Gamma',\gamma')$, where~$\Gamma'$ is an ordered group containing~$\Gamma$ and $\gamma' \in \Gamma'$.
 Similarly, for a \added{totally} ordered set~$\Gamma$ a \emph{ranger} is a pair $(\Gamma',\gamma')$, where~$\Gamma'$ is a \added{totally} ordered set containing~$\Gamma$ and $\gamma' \in \Gamma'$.
 Two rangers $(\Gamma_1,\gamma_1)$ and $(\Gamma_2,\gamma_2)$ are \emph{equivalent} if there is a third ranger $(\Gamma_3,\gamma_3)$ such that~$\Gamma_3$ is contained in $\Gamma_1$ and in~$\Gamma_2$ identifying~$\gamma_3$ with $\gamma_1$ and~$\gamma_2$, respectively.
\end{definition}

We see from this definition that up to isomorphism every equivalence class of rangers has a smallest representative.
If for a ranger $(\Gamma',\gamma')$ the element~$\gamma'$ is contained in~$\Gamma$, this smallest representative is $(\Gamma,\gamma')$.
\added{More generally, if $(\Gamma',\gamma')$ is a smallest representative of an equivalence class of rangers over an ordered group~$\Gamma$, then $\Gamma'$ is a quotient of $\Gamma \oplus \Z$, containing~$\Gamma$ and with a compatible ordering, and~$\gamma'$ is the image of $1 \in \Z$.
Similarly a ranger over a totally ordered set~$\Gamma$ can be identified either with an element of~$\Gamma$ or a total order on $\Gamma \cup \{*\}$ compatible with the ordering of~$\Gamma$.
From this description it is clear that the equivalence classes of rangers form a set, which we denote by~$\cR_\Gamma$.}

Unfortunately, in the case of ordered groups, there is no canonical multiplication on~$\cR_\Gamma$ extending the one on~$\Gamma$.
But we can multiply elements $\gamma \in \Gamma$ with a ranger $(\Gamma',\gamma')$ by setting
\[
 \gamma \cdot (\Gamma',\gamma') := (\Gamma',\gamma \gamma').
\]
In this way the set of rangers~$\cR_\Gamma$ is an ordered set with an action of~$\Gamma$.

For an ordered group~$\Gamma$ we denote by~$\Gamma_\Q$ its divisible hull and by~$|\Gamma|$ the underlying ordered set.

\begin{lemma} \label{rangers_comparison}
 We have natural identifications
 \[
  \cR_\Gamma \cong \cR_{\Gamma_\Q} \cong \cR_{|\Gamma_\Q|}.
 \]
\end{lemma}

\begin{proof}
 Every ranger of $\Gamma_\Q$ is naturally also a ranger of~$\Gamma$.
 If we start with a ranger $(\Gamma',\gamma')$ of~$\Gamma$, the ranger $(\Gamma'_\Q,\gamma')$ lies in the same equivalence class and comes from $\cR_{\Gamma_{\Q}}$.
 This proves the first isomorphism.
 
 A ranger of~$\Gamma_\Q$ is naturally also a ranger of~$|\Gamma_\Q|$.
 For a ranger $(\Gamma',\gamma')$ of~$|\Gamma_\Q|$ \added{we distinguish the cases $\gamma' \in |\Gamma_\Q|$ and $\gamma' \notin |\Gamma_\Q|$.
 In the first case~$\gamma'$ can be seen as an element of~$\Gamma_\Q \subset \cR_{\Gamma_\Q}$.
 In the second case we} define an ordering on the group $\Gamma_{\Q} \oplus (\gamma')^\Z$ in the following way.
 An element $\gamma (\gamma')^a$ is greater than $\delta (\gamma')^b$ if and only if either $a=b$ and $\gamma > \delta$ or $a \ne b$ and
 \[
  \left(\frac{\gamma}{\delta}\right)^{1/(b-a)} > \gamma'
 \]
 in~$\Gamma'$.
 This definition uses the divisibility of~$\Gamma_\Q$ and one checks that it satisfies the axioms for an ordered group.
 We conclude that $(\Gamma_\Q \oplus (\gamma')^\Z,\gamma')$ is a ranger of $\Gamma_\Q$ that is equivalent to $(\Gamma',\gamma')$ when viewed as a ranger of $|\Gamma_\Q|$.
\end{proof}

If $(\Gamma',\gamma')$ is a ranger in~$\cR_{|\Gamma|}$ we often just denote it by~$\gamma'$.
This means we implicitly assume an ordering on $\Gamma \cup \{\gamma'\}$.
In view of \cref{rangers_comparison} we can thus specify a ranger of~$\Gamma$ by giving a ranger~$\gamma'$ of~$|\Gamma_\Q|$.

Let us describe the different types of rangers.
A ranger $(\Gamma',\gamma')$ of~$\Gamma$ is called \emph{bounded} if there are $\gamma_1$ and~$\gamma_2$ in~$\Gamma$ such that
\[
 \gamma_1 \le \gamma' \le \gamma_2.
\]
Otherwise it is \emph{unbounded}.
There are only two equivalence classes of unbounded rangers: $0$, which satisfies $0 < \gamma$ for all $\gamma \in \Gamma$ and $\infty$ satisfying $\gamma < \infty$ for all $\gamma \in \Gamma$.
Every element~$\gamma$ of~$\Gamma$ determines a ranger $(\Gamma,\gamma)$.
We call these rangers \emph{principal}.
A nonprincipal bounded ranger $\gamma'$ divides~$\Gamma$ into a disjoint union
\[
 \Gamma = \underbrace{\{\gamma \in \Gamma \mid \gamma < \gamma'\}}_{\Gamma_{\gamma'}^-} \cup \underbrace{\{\gamma \in \Gamma \mid \gamma > \gamma'\}}_{\Gamma_{\gamma'}^+}.
\]
If~$\Gamma_{\gamma'}^-$ has a maximum or $\Gamma_{\gamma'}^+$ a minimum, we call~$\gamma'$ an \emph{infinitesimal ranger}.
Otherwise~$\gamma'$ is a \emph{cut ranger}.

The functoriality of rangers is a bit subtle.
For a surjection
\[
 p : \Delta \twoheadrightarrow \Gamma
\]
of totally ordered groups we obtain a natural map
\[
 p_* : \cR_\Delta \longrightarrow \cR_\Gamma
\]
via the following construction.
The kernel~$K$ of~$p$ is a convex subgroup of~$\Delta$.
For a ranger $(\Delta',\delta')$ of~$\Delta$ we consider the quotient~$\Gamma'$ of~$\Delta'$ by the convex subgroup~$K'$ generated by~$K$.
Then we have a diagram
\[
 \begin{tikzcd}
  \Delta'	\ar[r,twoheadrightarrow]							& \Gamma' = \Delta'/K'	\\
  \Delta	\ar[r,twoheadrightarrow]	\ar[u,hookrightarrow]	& \Gamma.				\ar[u,hookrightarrow]
 \end{tikzcd}
\]
Defining~$\gamma'$ to be the image of~$\delta'$ in~$\Gamma'$, we obtain a ranger $(\Gamma',\gamma')$ of~$\Gamma$ and one checks that this construction maps equivalent rangers to equivalent rangers.
Therefore, sending~$(\Delta',\delta')$ to $(\Gamma',\gamma')$ yields a well defined map~$p_*$.

For injections
\[
 \iota : \Delta \longrightarrow \Gamma
\]
of totally ordered groups the easier functoriality is contravariant.
Every ranger $(\Gamma',\gamma')$ of~$\Gamma$ is also a ranger of~$\Delta$ and this defines a map
\[
 \iota^* : \cR_\Gamma \longrightarrow \cR_\Delta.
\]
However, we can also construct a natural map $\cR_\Delta \to \cR_\Gamma$ if~$\Delta$ is a \emph{convex} subgroup.
This should be possible in greater generality but we won't need it here and refer to the aforementioned work in progress with Michael Temkin.

\begin{lemma} \label{rangers_convex_subgroup}
 Let $\iota:\Delta \to \Gamma$ be the inclusion of a convex subgroup of an ordered group~$\Gamma$.
 Then there is a unique inclusion of rangers
 \[
  \iota_*:\cR_\Delta \hookrightarrow \cR_\Gamma.
 \]
 with convex image that extends~$\iota$.
 More precisely, the image of~$\iota_*$ is given by those rangers $(\Gamma',\gamma')$ such that $\gamma' < \gamma$ for all $\gamma \in \Gamma_\Q$ with $\gamma > \Delta_\Q$ and $\gamma' > \gamma$ for all $\gamma \in \Gamma_\Q$ with $\gamma < \Delta_\Q$.
\end{lemma}

\begin{proof}
 By \cref{rangers_comparison} it suffices to construct a map
 \[
  \cR_{|\Delta_\Q|} \longrightarrow \cR_{|\Gamma_\Q|},
 \]
 We map a ranger $\delta'$ of $|\Delta_\Q|$ to
 \[
  (\Gamma_\Q \cup \{\delta'\},\delta')
 \]
 with the ordering defined by
 \[
  \gamma < \delta'	\iff \gamma < \Delta ~\text{or}~[\gamma \in \Delta~\text{and}~\gamma<\delta']
 \]
 for an element $\gamma \in \Gamma_\Q$.
 It is clear that the resulting map on rangers has the prescribed image and that it is unique with this property.
\end{proof}

\begin{lemma}
 For an inclusion
 \[
  \iota : \Delta \hookrightarrow \Gamma
 \]
 of a convex subgroup~$\Delta$ of a totally ordered group~$\Gamma$, the composition
 \[
  \cR_\Delta \overset{\iota_*}{\longrightarrow} \cR_\Gamma \overset{\iota^*}{\longrightarrow} \cR_\Delta
 \]
 is the identity.
\end{lemma}

\begin{proof}
 By \cref{rangers_comparison} we may assume that~$\Gamma$ and~$\Delta$ are divisible.
 Moreover, by the same lemma we know that it is enough to consider rangers of the underlying ordered sets.
 For principal rangers the statement is clear.
 
 The map $\iota^*$ maps a nonprincipal minimal ranger $(\Delta \cup \{\delta\},\delta')$ of~$|\Delta|$ to $(\Gamma \cup \{\delta'\},\delta')$ with the ordering described as in the proof of \cref{rangers_convex_subgroup}.
 Applying~$\iota^*$ we get the same ranger $(\Gamma \cup \{\delta'\},\delta')$ but viewed as a ranger of~$|\Delta|$.
 As such it is equivalent to $(\Delta \cup \{\delta'\},\delta')$, which proves the assertion.
\end{proof}

Suppose~$\Gamma$ is a totally ordered group that possesses a convex subgroup~$\Delta$ such that the quotient $\Gamma/\Delta$ has rank~$1$ (i.e., $\Gamma$ is microbial).
We embed $\Gamma/\Delta$ into $\R_{>0}$ and thus obtain a homomorphism
\[
 \Gamma \longrightarrow \Gamma/\Delta \hookrightarrow \R_{>0} = (0,\infty).
\]
We extend this to a map
\[
 \pi : \cR_\Gamma \longrightarrow [0,\infty]
\]
in the following way.

We take a bounded ranger $(\Gamma',\gamma')$.
Then also~$\Gamma'$ has a convex subgroup~$\Delta'$ such that $\Gamma'/\Delta'$ has rank~$1$ and $\Delta \subseteq \Delta'$.
The resulting homomorphism
\[
 \Gamma/\Delta \longrightarrow \Gamma'/\Delta'
\]
is an inclusion of totally ordered groups of rank~$1$ (using the boundedness of the ranger~$\gamma'$).
The inclusion $\Gamma/\Delta \hookrightarrow \R_{>0}$ extends uniquely to $\Gamma'/\Delta'$ and we obtain a homomorphism
\[
 \Gamma' \twoheadrightarrow \Gamma'/\Delta' \hookrightarrow \R_{>0}.
\]
extending $\Gamma \to \R_{>0}$.
Now we define $\pi(\Gamma',\gamma')$ to be the image of~$\gamma'$ in $\R_{>0}$ under this homomorphism.

Concerning the unbounded rangers~$0$ and~$\infty$, we just send~$0$ to~$0$ and~$\infty$ to~$\infty$.
One readily checks that this construction does not depend on the representative $(\Gamma',\gamma')$ of an equivalence class of rangers and that
\[
 \gamma'_1 \le \gamma'_2	\implies	\pi(\gamma'_1) \le \pi(\gamma'_2).
\]
Moreover,~$\pi$ is compatible with the action of~$\Gamma$ in the sense that for $\gamma \in \Gamma$ and $\gamma' \in \cR_\Gamma$ we have
\[
 \pi(\gamma \gamma') = \pi(\gamma)\pi(\gamma').
\]
In view of \cref{rangers_comparison} we can also describe~$\pi$ in terms of rangers $\cR_{|\Gamma_\Q|}$ of the totally ordered set~$|\Gamma_\Q|$.
A non-principal ranger $\gamma'$ of $|\Gamma_\Q|$ defines a decomposition of~$|\Gamma_\Q|$ into a disjoint union
\[
 \{\gamma \in \Gamma_\Q \mid \gamma < \gamma'\} < \{\gamma \in \Gamma_\Q \mid \gamma > \gamma'\}.
\]
of (possibly empty) subsets of~$\Gamma_\Q$.
Since~$\pi$ respects the ordering,
\[
 \pi(\{\gamma \in \Gamma_\Q \mid \gamma < \gamma'\}) \le \pi(\{\gamma \in \Gamma_Q \mid \gamma > \gamma'\})
\]
and these sets are either disjoint or contain one common element.
Moreover, as $\pi(\Gamma_\Q)$ is dense in~$\R_{<0}$,
\[
 \supp \{\gamma \in \Gamma_\Q \mid \gamma < \gamma'\} = \inf \{\gamma \in \Gamma_\Q \mid \gamma > \gamma'\}
\]
(in case both of the sets are nonempty).
Then $\pi(\gamma')$ equals this infimum/supremum.
If one of the sets is empty we obtain $\pi(\gamma') = 0$ or $\pi(\gamma') = \infty$ according to whether the right or the left hand set is empty.

\begin{proposition} \label{fibers_pi}
 Suppose~$\Gamma$ is microbial.
 The map
 \[
  \pi : \cR_\Gamma \longrightarrow [0,\infty]
 \]
 is surjective.
 Moreover, for $\gamma \in \Gamma_\Q$ the fibre over $\pi(\gamma)$ equals $\gamma \cdot \cR_\Delta$ (where $\cR_\Delta$ is viewed as a subset of $\cR_\Gamma$ via \cref{rangers_convex_subgroup}) and for $r \in [0,\infty] \setminus \pi(\Gamma_\Q)$, the fibre consists of a single point $(\Gamma',\gamma') \in \cR_\Gamma \setminus \Gamma$.
\end{proposition}

\begin{proof}
 By \cref{rangers_comparison} it suffices to check the corresponding statement for~$\cR_{|\Gamma_\Q|}$.
 For the surjectivity of~$\pi$ we just need to check that the points in $[0,\infty] \setminus \pi(\Gamma_\Q)$ lie in the image.
 But such a point~$r$ defines a ranger~$\gamma'$ of $|\Gamma_\Q|$ with ordering defined by
 \[
  \gamma' < \gamma	\iff	r < \pi(\gamma)
 \]
 for $\gamma \in |\Gamma_\Q|$ and we have 
 \[
  \pi(\gamma') = r.
 \]
 From this description it is also clear that the fibre over~$r$ consists of the single ranger~$\gamma'$.
 
 Now we take an element $\gamma \in \Gamma$ and a ranger $\delta' \in \cR_{|\Delta_\Q|}$.
 By \cref{rangers_convex_subgroup}, $\delta'$ can be viewed as a ranger of $|\Gamma_\Q|$ satisfying $\delta' < \gamma_1$ for all $\gamma_1 \in \Gamma_\Q$ with $\gamma_1 > \Delta_\Q$ and $\delta' > \gamma_2$ for all $\gamma_2 \in \Gamma_\Q$ with $\gamma_2 < \Delta_\Q$.
 Then
 \[
  \pi(\{\gamma_2 \in \Gamma_\Q \mid \gamma_2 < \Delta_\Q\}) \le \pi(\delta') \le \pi(\{\gamma_1 \in \Gamma_\Q \mid \gamma_1 > \Delta_\Q\}.
 \]
 The union of the left hand and the right hand subset of $\R_{>0}$ equals $\pi(\Gamma_\Q \setminus \Delta_\Q)$, which is dense in $\R_{>0}$ as all of $\Delta_\Q$ is mapped to the single point~$\{1\}$.
 Moreover,
 \[
  \pi(\{\gamma_2 \in \Gamma_\Q \mid \gamma_2 < \Delta_\Q\}) \le 1 \le \pi(\{\gamma_1 \in \Gamma_\Q \mid \gamma_1 > \Delta_\Q\}.
 \]
 This implies $\pi(\delta') = 1$ and hence
 \[
  \pi(\gamma \delta') = \pi(\gamma) \pi(\delta') = \pi(\gamma).
 \]
 We conclude that $\gamma \cR_\Gamma \subseteq \pi^{-1}(\pi(\gamma))$.
 On the other hand suppose that $\pi(\gamma') = \pi(\gamma)$ for some ranger $\gamma' \in \cR_{|\Gamma_\Q|}$.
 We set $\delta' = \gamma^{-1}\gamma'$ and want to show $\delta' \in \cR_{|\Delta_\Q|}$ via the description of the image of $\cR_\Delta \to \cR_\Gamma$ from \cref{rangers_convex_subgroup}.
 Let $\gamma_1$ be an element of~$\Gamma_\Q$ satisfying $\gamma_1 > \Delta_\Q$.
 Then
 \[
  \pi(\gamma_1) > 1 = \pi(\delta').
 \]
 This implies $\gamma_1 > \delta'$.
 Similarly we have $\gamma_2 < \delta'$ for all $\gamma_2 \in \Gamma_\Q$ such that $\gamma_2 < \Delta_\Q$.
 By \cref{rangers_convex_subgroup} this implies $\delta' \in \cR_{|\Delta_\Q|}$, whence
 \[
  \gamma' \in \gamma \cR_\Delta.
 \]
\end{proof}

\section{Seminorms} \label{section_seminorms}

Having introduced rangers in \cref{section_rangers} we can now study seminorms that take values in rangers.
Let $A$ be a ring together with a valuation.
For an element $a \in A$ we denote its value by $|a|$ and write~$\Gamma$ for the value group of the valuation.

An \emph{ordered $\Gamma$-set} is an ordered set~$\bar{\Gamma}$ together with an action of~$\Gamma$ preserving the ordering.
By this we mean that for $\gamma_1 < \gamma_2$ in $\Gamma$ and $\bar{\gamma}_1 \le \bar{\gamma}_2$ in $\bar{\Gamma}$ we get
\[
 \gamma_1 \bar{\gamma}_1 < \gamma_2 \bar{\gamma}_2
\]
and similarly for $\gamma_1 \le \gamma_2$ and $\bar{\gamma}_1 < \bar{\gamma}_2$.
We call~$\bar{\Gamma}$ \emph{complete} if all infima and all suprema of subsets of~$\bar{\Gamma}$ are attained in~$\bar{\Gamma}$.

Our most important example of such a $\Gamma$-set is given by the rangers~$\cR_\Gamma$.
But we will also consider $\bar{\Gamma} = [0,\infty]$ in case~$\Gamma \subseteq \R_{>0}$.

For a complete ordered $\Gamma$-set~$\bar{\Gamma}$ we extend the action of $\Gamma$ to $\Gamma \cup \{0\}$ by setting
\[
 0 \cdot x := 0,
\]
where the zero on the right is defined as the infimum of the set $\bar{\Gamma}$ (which exists as~$\bar{\Gamma}$ is complete).
Note that in this way we also choose to set $0 \cdot \infty = 0$.

We generalize the notion of a $K$-seminorm in the following way

\begin{definition}
 A \emph{$\bar{\Gamma}$- seminorm} on an $A$-module~$V$ is a map
 \[
  |\cdot| : V \longrightarrow \bar{\Gamma}
 \]
 satisfying
 \begin{enumerate}[(i)]
  \item $|a x| = |a| \cdot |x|$ for all $a \in A$ and $x \in V$,
  \item	$|x + y| \le \max\{|x|,|y|\}$ for all $x,y \in V$.
 \end{enumerate}
\end{definition}

The same proof as in the classical case shows that $|x + y| = \max \{|x|,|y|\}$ in case $|x| \ne |y|$.

\subsection{The projective tensor product seminorm}

For a homomorphism $A \to B$ of valued rings and a $\bar{\Gamma}$-seminormed $A$-module~$V$ we want to study the tensor product
\[
 V \otimes_A B.
\]
If~$\bar{\Gamma}$ is complete and the action of~$\Gamma$ on~$\bar{\Gamma}$ extends to an order preserving action of the value group~$\Gamma_B$ of~$B$, we can equip $V \otimes_A B$ with the following seminorm with values in~$\bar{\Gamma}$:
For $w \in V \otimes_A B$ we set
\[
 |w| := \inf_{w=\sum_i b_i \otimes v_i}\{\max_i |b_i|\cdot |v_i|\},
\]
where the infimum runs over all presentations of~$w$ as a finite sum of elements of the form $b \otimes v$ for $b \in B$.
This seminorm is called \emph{projective tensor product seminorm}.

The aim of this subsection is to prove under some assumptions on~$A$ and~$B$ that
\[
 V \longrightarrow V \otimes_A B
\]
is an isometric embedding.
Statements like this have been proved in the literature for Banach spaces over a nonarchimedean field (with a rank one valuation).
We start by extending this result to $[0,\infty]$-seminormed vector spaces.

\begin{lemma} \label{tensor_product_classical}
 Let $L/K$ be an extension of nonarchimedean fields.
 For a $[0,\infty]$-seminormed $K$-vector space~$V$ the homomorphism
 \[
  V \longrightarrow V \otimes_K L
 \]
 is an isometric embedding. 
\end{lemma}

\begin{proof}
 It is easy to check that
 \[
  V_f := \{v \in V \mid |v| < \infty\}
 \]
 ($f$ for finite) is a subvector space of~$V$.
 For $v \in V_f$ the seminorm takes values in $\R_{\ge 0} = [0,\infty)$, so $V_f$ is a seminormed $K$-vector space in the classical sense.
 Moreover, it is easy to check that the restriction of the seminorm on $V \otimes_K L$ to $V_f \otimes_K L$ coincides with the projective tensor product seminorm on $V_f \otimes_K L$.
 Therefore we conclude by the classical result (\cite{PerGar-Schik-VS}, Corollary~10.2.10) that
 \[
  |v_f \otimes 1| = |v_f|
 \]
 for all $v_f \in V_f$.
 Suppose $v \notin V_f$ and let
 \[
  v \otimes 1 = \sum_i b_i \otimes v_i
 \]
 be a presentation of~$v$ in terms of $b_i \in B$ and $v_i \in V$.
 Not all of the vectors~$v_i$ can be contained in~$V_f$ as otherwise $v \otimes 1 \in V_f \otimes_K L$.
 Hence,
 \[
  \max_i \{|b_i| \cdot |v_i|\} = \infty.
 \]
 Taking the infimum over all presentations of $v \otimes 1$ gives
 \[
  |v \otimes 1| = \infty = |v|.
 \]
\end{proof}

We use induction on the rank to prove it in more generality.

\begin{proposition} \label{tensor_product_vs}
 Let $(L,L^+)/(K,K^+)$ be an algebraic extension of valued fields of finite rank and $V$ a $K$-vector space equipped with a $K$-seminorm $V \to \cR_{\Gamma_K}$.
 Then the natural embedding
 \[
  V \longrightarrow V \otimes_K L,
 \]
 where $V \otimes_K L$ is endowed with the projective tensor product seminorm, is isometric, i.e.,
 \[
  |v \otimes 1| = |v|
 \]
 for all $v \in V$.
\end{proposition}

\begin{proof}
 First of all we may assume that~$K$ and~$L$ are complete and~$V$ is a Banach space as completion does not change the (semi)norms and behaves well with the projective tensor product seminorm.
 We assumed that $L/K$ is algebraic because then the rangers for $\Gamma_K$ and~$\Gamma_L$ are the same.
 This is used implicitly in the constructions below.
 Moreover, we set $\Gamma_\Q := \Gamma_K \otimes \Q = \Gamma_L \otimes \Q$.
 
 We start with the trivially valued case.
 This does not imply that the seminorm on~$V$ needs to be trivial.
 It takes values in
 \[
  \cR_{\{1\}} = \{0,1,\infty\} \subseteq [0,\infty].
 \]
 Therefore, the statement is a special case of \cref{tensor_product_classical}.
 
 We now assume that the rank of~$\Gamma_K$ is at least~$1$.
 Let $\Delta$ be the maximal proper convex subgroup of~$\Gamma_K$, such that $\Gamma_K/\Delta$ has rank~$1$.
 Then we choose an embedding $\Gamma_K/\Delta \hookrightarrow \R_{>0}$.
 Composing with the valuation on~$K$ gives a real valuation
 \[
  |\cdot|^\circ : K \overset{|\cdot|}{\longrightarrow} \Gamma_K \twoheadrightarrow \Gamma_K/\Delta \hookrightarrow \R_{\ge 0}.
 \]
 The corresponding valuation ring of~$K$ is the set~$K^\circ$ of power bounded elements with maximal ideal $K^{\circ \circ}$.
 We write~$K^\succ$ for the residue field $K^\circ/K^{\circ\circ}$ and similarly for~$L$.
 Moreover, we compose the seminorm on~$V$ with the map
 \[
  \pi : \cR_\Gamma \longrightarrow [0,\infty]
 \]
 constructed in \cref{section_rangers} to obtain a seminorm
 \[
  |\cdot|^\circ : V \overset{|\cdot|}{\longrightarrow} \cR_{\Gamma_K} \overset{\pi}{\longrightarrow} [0,\infty]
 \]
 with respect to the valuation $|\cdot|^\circ$ on~$K$.
 
 Let $v_0$ be an element of~$V$.
 We want to show that
 \[
  |v_0 \otimes 1| = |v_0|.
 \]
 By \cref{tensor_product_classical} we know that
 \[
  \pi(|v_0 \otimes 1|) = |v_0 \otimes 1|^\circ = |v_0|^\circ = \pi(|v_0|).
 \]
 In case $|v_0| \notin \Gamma_\Q$ \cref{fibers_pi} tells us that this already implies 
 \[
  |v_0 \otimes 1| = |v_0|.
 \]
 If $|v_0| \in \Gamma_\Q$ we only know that
 \[
  |v_0 \otimes 1| \in |v_0| \cR_\Delta
 \]
 and we have to work a bit more.
 We set  $r := \pi(|v_0|)$ and consider the $K^\circ$-submodule
 \[
  V_r := \{v \in V \mid |v|^\circ \le r\}.
 \]
 of~$V$.
 It fits into the following diagram
 \[
  \begin{tikzcd}
   V						\ar[r]							& V \otimes_K L		\\
   V_r						\ar[r]	\ar[u,hookrightarrow]	& V_r \otimes_{K^\circ} L^\circ.	\ar[u,hookrightarrow]	\\
  \end{tikzcd}
 \]
 The upper horizontal map is isometric with respect to the seminorm $|\cdot|^\circ$ by \cref{tensor_product_classical}.
 For $v \in V_r$ we have inequalities (where we indicate in which spaces we are taking the seminorm)
 \[
  |v \otimes 1|_{V \otimes_K L}^\circ \le |v \otimes 1|_{V_r \otimes_{K^\circ} L^\circ}^\circ \le |v|_{V_r}^\circ = |v|_V^\circ.
 \]
 Since $V \to V \otimes_K L$ is isometric, all of these inequalities are equalities.
 Therefore,
 \[
  V_r \longrightarrow V_r \otimes_{K^\circ} L^\circ
 \]
 is an isometric embedding as well.

 The valuation~$|\cdot|$ on~$K$ can be written as a composition of the valuation~$|\cdot|^\circ$ with a valuation~$|\cdot|^\succ$ on the residue field $K^\succ = K^\circ/K^{\circ \circ}$ of~$K^\circ$.
 Precomposing~$|\cdot|^\succ$ with the natural projection $K^\circ \to K^\succ$ we obtain a valuation of~$K^\circ$ that we again denote by~$|\cdot|^\succ$.
 Explicitly for an element $a \in K^\circ$ it is given by
 \[
  |a|^\succ = \begin{cases}
               |a|	& |a| \in \Delta,	\\
               0	& \text{else}.
              \end{cases}
 \]
 We now do a similar construction in terms of seminorms.
 Remember that we defined $r = \pi(|v_0|)$.
 We define the seminorm $|\cdot|^\succ$ of $V_r$ over $(K^\circ,|\cdot|^\succ)$ as follows
 \[
  |v|^\succ := \begin{cases}
                |v|	& |v| \in |v_0| \cR_\Delta,	\\
                0	& \text{else}.
               \end{cases}
 \]
 The axioms for a seminorm (especially the triangle inequality) are satisfied because if $v \in V_r$, then $|v|$ only takes values less or equal to $|v_0| \Delta$, so $|v| \notin |v_0| \Delta$ implies $|v| < |v_0| \Delta$.
 
 We now consider the diagram
 \[
  \begin{tikzcd}
   V_r						\ar[r]	\ar[d,twoheadrightarrow]	& V_r \otimes_{K^\circ} L^\circ						\ar[d,twoheadrightarrow]	\\
   V_r/K^{\circ \circ} V_r	\ar[r]								& V_r/K^{\circ \circ}V_r \otimes_{K^\succ} L^\succ,
  \end{tikzcd}
 \]
 where all spaces are equipped with valuations respectively seminorms defined by~$|\cdot|^\succ$ and its projective tensor product seminorm.
 The vertical arrows are obtained by taking quotients by subspaces where the seminorm vanishes.
 They are thus isometric.
 The lower horizontal map is a map of seminormed vector spaces over the valued field~$K^\succ$, whose rank is one less than the rank of~$K$.
 Hence, we can assume by induction that the lower horizontal map is isometric.
 This implies that also the upper horizontal map is isometric.
 
 Finally we piece the two parts about~$|\cdot|^\circ$ and~$|\cdot|^\succ$ together.
 We have already observed that
 \[
  |v_0 \otimes 1| \in |v_0| \Delta.
 \]
 We thus have
 \[
  |v_0| = |v_0|^\succ, \qquad |v_0 \otimes 1| = |v_0 \otimes 1|^\succ
 \]
 and by the last paragraph
 \[
  |v_0 \otimes 1| = |v_0 \otimes 1|^\succ = |v_0|^\succ = |v_0|.
 \]
\end{proof}

\begin{remark}
 \cref{tensor_product_vs} also holds true with the same proof if we only require $\Gamma_K \otimes \Q = \Gamma_L \otimes \Q$.
 In addition to algebraic extensions this applies for instance to the case where~$L$ is the residue field of a point of type~$2$ or~$4$ on a rigid analytic curve over~$K$.
 In fact we would expect the same result for arbitrary extensions of valued fields but this would require a more thorough treatment of the functoriality of the construction of rangers.
 More precisely one would have to construct an embedding $\cR_{\Gamma_K} \hookrightarrow \cR_{\Gamma_L}$.
\end{remark}

\begin{corollary} \label{tensor_product_Huber_pairs}
 Let $(A,A^+) \to (B,B^+)$ be a local homomorphism of local Huber pairs such that the extension of residue fields $k_A/k_B$ is algebraic.
 \added{We assume that the value group~$\Gamma_A$ of the valuation on~$A$ has finite rank}.
 For any $\cR_{\Gamma_A}$-seminormed $A$- module~$M$ the homomorphism
 \[
  M \longrightarrow M \otimes_A B
 \]
 is isometric.
\end{corollary}

\begin{proof}
 Consider the diagram
 \[
  \begin{tikzcd}
   M		\ar[r]	\ar[d]	& M \otimes_A B				\ar[d]	\\
   M/\m_A M	\ar[r]			& M/\m_A \otimes_{k_A} k_B.
  \end{tikzcd}
 \]
 The quotient $M/\m_A M$ is naturally a seminormed $k_A$-vector space as $\m_A$ is contained in (and in fact equals) the kernel of the valuation of~$A$.
 In this way the left vertical map is isometric.
 The right hand map is then also isometric both spaces being equipped with the projective tensor product seminorm.
 Moreover, the lower horizontal map is isometric by \cref{tensor_product_vs}.
 This implies that the upper horizontal map is isometric.
\end{proof}

\subsection{The adic seminorm} \label{section_adic_seminorm}

In this subsection we consider a local Huber pair $(A,A^+)$ and equip~$A^+$ with the valuation corresponding to the closed point of $\Spa(A,A^+)$.
We let~$\Gamma$ denote the value group of this valuation.

\begin{definition}
 For a local Huber pair $(A,A^+)$ and an $A^+$-module~$M$ we define the \emph{adic seminorm} by
 \[
  |x|^{\ad} := \inf_{\substack{a^+ \in A^+ \\ x \in a^+ M}} |a^+|_A \in \cR_\Gamma.
 \]
\end{definition}

The nice thing about the adic seminorm is that it behaves well with respect to base change.
If $(A,A^+) \to (B,B^+)$ is a local homomorphism of local Huber pairs, we can equip $M \otimes_A B$ with the adic seminorm with respect to~$B$.
But if the residue field extension $k_B/k_A$ is algebraic, we can also consider the projective tensor product seminorm on $M \otimes_A B$.
To distinguish these two seminorms we will write $|x|^\otimes$ for the tensor product seminorm of $x \in M \otimes_A B$.

\begin{lemma} \label{adic_seminorm_bc}
 Let $(A,A^+) \to (B,B^+)$ be a local homomorphism of local Huber pairs with algebraic residue field extension and $M$ an~$A^+$ module equipped with the adic seminorm.
 Then the adic seminorm of
 \[
  M_B := M \otimes_{A^+} B^+
 \]
 as a $B^+$-module coincides with the projective tensor product seminorm.
\end{lemma}

\begin{proof}
 First of all note that all the seminorms showing up in the statement take values less or equal to~$1$ as they are obtained by taking an infimum of values of elements of~$A^+$ or~$B^+$.
 We first claim that for $m \in M$ and $b^+ \in B^+$ we have the inequality
 \[
  |m \otimes b^+|^{\ad} \le |m|^{\ad} \cdot |b|^+.
 \]
 In order to see this we take $a^+ \in A^+$ such that $m \in a^+M$.
 So there is $m' \in M$ such that $m = a^+ m'$.
 This implies
 \[
  m \otimes b^+ = (a^+m') \otimes b^+ = (a^+b^+)(m' \otimes 1).
 \]
 Taking the adic seminorm we obtain
 \[
  |m \otimes b^+|^{\ad} \le |a^+| \cdot |b^+| \cdot |m' \otimes 1|^{\ad} \le |a^+| \cdot |b^+|.
 \]
 This holds for all $a^+ \in A^+$ such that $m \in a^+ M$ and the claim follows.
 
 Let $x$ be an element of~$M_B$.
 We now want to show that
 \begin{equation} \label{ad_leq_tensor}
  |x|^{\ad} \le |x|^\otimes.
 \end{equation}
 To this end we choose a presentation of~$x$ as a sum of tensors:
 \[
  x = \sum_i m_i \otimes b_i^+.
 \]
 for $m_i \in M$ and $b_i^+ \in B^+$.
 Then
 \[
  |x|^{\ad} \le \max_i \{|m_i \otimes b_i^+|^{\ad}\} \le \max_i \{|m_i|^{\ad} \cdot |b_i^+|\},
 \]
 where the second inequality is justified by the claim above.
 Taking the infimum over all presentations of~$x$ as a sum of tensors we obtain~(\ref{ad_leq_tensor}).
 
 Finally we show that
 \begin{equation} \label{tensor_leq_ad}
  |x|^\otimes \le |x|^{\ad}.
 \end{equation}
 Let $b^+ \in B^+$ be an element such that $x \in b^+M_B$.
 Writing $x = b^+ x'$ for some $x' \in M_B$, we obtain
 \[
  |x|^\otimes = |b^+ x'|^\otimes \le |b^+| \cdot |x'|^\otimes \le |b^+|.
 \]
This is also true for the infimum over all such~$b^+$ and we conclude that~(\ref{tensor_leq_ad}) holds.
\end{proof}

\section{The K\"ahler seminorm} \label{section_Kaehler}

\subsection{The K\"ahler seminorm for local Huber pairs} \label{section_Kaehler_local}

Fix $(k,k^+)$ as in \cref{section_log_diff_adic}, i.e., we are in either of the following three situations:
\begin{itemize}
 \item	the residue characteristic of~$k^+$ is~$0$,
 \item	$k$ is perfect and $k^+ = k$,
 \item	$k$ is algebraically closed.
\end{itemize}
Moreover, throughout this subsection $(A,A^+)$ is a local Huber pair over $(k,k^+)$ equipped with the discrete topology.
We denote by $\m_A$ the maximal ideal of~$A$ and by~$k_A$ the residue field $k_A = A/\m_A$.
The field $k_A$ carries a valuation~$|\cdot|$ defined by the valuation ring $k_A^+ = A^+/\m_A$.
We denote its value group by~$\Gamma$.

\begin{definition}
 We define the \emph{K\"ahler seminorm}
 \[
  |\cdot|_\Omega : \Omega_A \longrightarrow \cR_\Gamma
 \]
on $\Omega_A$ by
 \[
  |\omega|_{\Omega} := \inf_{\omega = \sum f_i dg_i} \max_i \{|f_i|_A |g_i|_A\},
 \]
 where the infimum is over all representations of~$\omega$ as a finite sum $\omega = \sum_i f_i dg_i$.
\end{definition}

If the infimum is not achieved, its value is always a ranger that is not contained in~$\Gamma$:

\begin{lemma}
 Let $\Gamma$ be an abelian ordered group and $\cR_\Gamma \supseteq \Gamma$ its set of rangers.
 For a subset $M \subseteq \Gamma$ the infimum $\inf M \in \cR_\Gamma$ is contained in~$\Gamma$ if and only if~$M$ contains a minimal element.
\end{lemma}

\begin{proof}
 If~$M$ contains a minimal element~$m$, it is clear that $m = \inf M \in \Gamma$.
 Otherwise we construct for every $\gamma \in \Gamma$ with $\gamma < M$ and element $(\Gamma',\gamma') \in \cR_\Gamma$ with
 \[
  \gamma < \gamma' < M.
 \]
 Indeed, this can be achieved by taking
 \[
  \Gamma' := \Gamma \oplus \Z
 \]
 with the lexicographical ordering and
 \[
  \gamma' = (\gamma,1).
 \]
\end{proof}

This is different from the real valued case treated in \cite{Tem16}, \S~5 where it can happen that the infimum is not achieved but its value is contained in~the value group.
However, this should not be considered as a bug but as a feature as now the information on whether the infimum is achieved or not is contained in its value.

By \cref{Omega_injective} we can consider $\Omega^{\log}_{(A,A^+)}$ as an $A^+$-submodule of $\Omega_A$.

\begin{lemma} \label{unit_ball}
 We have
 \[
  \Omega^{\log}_{(A,A^+)} = \{\omega \in \Omega_A \mid |\omega|_{\Omega} \le 1\}.
 \]
\end{lemma}

\begin{proof}
 A general element of $\Omega^{\log}_{(A,A^+)}$ can be written as a finite sum of the form
 \[
  \omega = \sum_i f_i dg_i/g_i + \sum_j a_j db_j
 \]
 with $f_i,a_j,b_j \in A^+$ and $g_i \in A^+ \setminus \m_A$.
 Its Kähler seminorm satisfies
 \[
  |\omega|_{\Omega} \le \max_{i,j}\{|f_i|,|a_j| \cdot |b_j|\} \le 1.
 \]
 Now take $\omega \in \Omega_A^1$ with $|\omega|_{\Omega} \le 1$.
 By definition there is a representation $\omega = \sum_i f_i dg_i$ with
 \[
  \max_i \{|f_i|_A |g_i|_A\} \le 1,
 \]
 i.e., $|f_i|_A |g_i|_A \le 1$ for all~$i$.
 So $f_i g_i \in A^+$.

 In case $g_i \notin \m_A$ we write $g_i = g'_i/g''_i$ with $g'_i, g''_i \in A^+ \setminus \m_A$.
 Then
 \[
  f_i dg_i = f_i g_i \frac{dg'_i}{g'_i} - f_i g_i \frac{dg''_i}{g''_i} \in \Omega^{\log}_{(A,A^+)}.
 \]
 Suppose now that $g_i \in \m_A$.
 If $f_i \in \m_A$ as well, then in particular, $f_i$ and $g_i$ are both elements of~$A^+$.
 Hence, $f_i dg_i \in \Omega_{A^+} \subset \Omega^{\log}_{(A,A^+)}$ ($\Omega_{A^+}$ can be viewed as a submodule of $\Omega^{\log}_{(A,A^+)}$ by \cref{Omega_injective}).
 Finally, if $g_i \in \m_A$ but $f_i \notin \m_A$, we write
 \[
  f_i dg_i = d(f_i g_i) - g_i df_i.
 \]
 The first term is in $\Omega^1_{A^+} \subset \Omega^{\log}_{(A,A^+)}$ and the second term in $\Omega^{\log}_{(A,A^+)}$ by the same reasoning as above.
 We conclude that $\omega \in \Omega^{\log}_{(A,A^+)}$.
\end{proof}

\begin{lemma} \label{characterization_Kaehler}
 The K\"ahler seminorm is the maximal $A$-seminorm $|\cdot|:\Omega_A \to \cR_\Gamma$ with $|\Omega_{(A,A^)}^{\log}| \le 1$ and $|dx| = 0$ for $x \in \m_A$.
\end{lemma}

\begin{proof}
 We already know from \cref{unit_ball} that the K\"ahler seminorm is less or equal to one on logarithmic differentials.
 It is also clear from the definition that $|dx|_{\Omega} = 0$ for $x \in \m_A$.
 It remains to show the maximality.
 Let $|-|$ be a seminorm such that $|\Omega_A^{\log}| \le 1$ and $|dx| = 0$ for $x \in \m_A$.
 Let $\omega \in \Omega_A$ and pick a representation $\omega = \sum_i f_i dg_i$.
 For every~$i$ such that $g_i \notin \m_A$ take $g'_i, g''_i \in A^+ \setminus \m_A$ such that $g_i = g'_i/g''_i$.
 Then
 \[
 f_i dg_i = f_i g_i(\frac{dg'_i}{g'_i} - \frac{dg''_i}{g''_i}).
 \]
 For $i$ with $g_i \in \m_A$ we have $|f_i dg_i| = 0$.
 Hence, by the strong triangle inequality,
 \[
  |\omega| \le \max_{i,g_i \notin \m_A} \{|f_i g_i|_A |\frac{dg'_i}{g'_i}|,|f_i g_i|_A |\frac{dg''_i}{g''_i}|\}.
 \]
 By our assumption $|\frac{dg'_i}{g'_i}| \le 1$ and $|\frac{dg''_i}{g''_i}| \le 1$, whence
 \[
  |\omega| \le \max_{i,g_i \notin \m_A} \{|f_i|_A |g_i|_A\} = \max_i \{|f_i|_A |g_i|_A\}.
 \]
 Since this holds for all representations $\omega = \sum_i f_i g_i$, we obtain $|\omega| \le |\omega|_{\Omega}$.
\end{proof}

Next we want to study the log differentials $\Omega_{(A,A^+)}^{\log}$.
We can consider the adic seminorm on $\Omega_{(A,A^+)}^{\log}$ from \cref{section_adic_seminorm}.
On the other hand, we have an inclusion $\Omega^{\log}_{(A,A^+)} \hookrightarrow \Omega_A$ (see \cref{Omega_injective}).
We thus obtain a seminorm on $\Omega^{\log}_{(A,A^+)}$ by restricting the Kähler seminorm to $\Omega^{\log}_{(A,A^+)}$.

\begin{lemma} \label{Kaehler_adic}
 For $\omega \in \Omega^{\log}_{(A,A^+)}$ we have $|\omega|_{\Omega} = |\omega|^{\ad}$
\end{lemma}

\begin{proof}
 By \cref{characterization_Kaehler} it suffices to show that $|\Omega^{\log}_{(A,A^+)}|^{\ad} \le 1$, $|dx|^{\ad} = 0$ for $x \in \m_A$, and $|\omega|_{\Omega} \le |\omega|^{\ad}$ for all $\omega \in \Omega^{\log}_{(A,A^+)}$.
 The first assertion is obvious as $|\Omega^{\log}_{(A,A^+)}|^{\ad} \le |A^+|_A$.
 For the second one take $x \in \m_A$ and $a^+ \in A^+ \setminus \m_A$.
 Then~$x$ is divisible by~$a^+$ and
 \[
  dx = d(a^+ \cdot \frac{x}{a^+}) = a^+d (\frac{x}{a^+}) + \frac{x}{a^+} da^+ = a^+(d(\frac{x}{a^+}) + \frac{x}{(a^+)^2} da^+),
 \]
 i.e., $dx \in a^+\Omega^{\log}_{(A,A^+)}$.
 By the definition of the adic seminorm, this means $|dx| \le |a^+|$.
 As~$a^+$ was arbitrary, this implies $|dx| = 0$.

 Let us now prove the last assertion.
 Take $\omega \in \Omega^{\log}_{(A,A^+)}$ and $a^+ \in A^+$ with $\omega \in a^+ \Omega^{\log}_{(A,A^+)}$.
 So there is a representation $\omega = \sum_i a^+ f_i dg_i/g_i + \sum_j a^+ b_j dc_j$ with $f_i,b_j,c_j \in A^+$ and $g_i \in A^+ \setminus \m_A$.
 Then
 \[
  |\omega|_{\Omega} \le \max_{i,j} \{|a^+|_A \cdot |f_i|_A, |a^+|_A \cdot |b_j|_A \cdot |c_j|_A\} \le |a^+|_A.
 \]
 Since this holds for all $a^+ \in A^+$ with $\omega \in a^+ \Omega^{\log}_{(A,A^+)}$, we obtain
 \[
  |\omega|_{\Omega} \le |\omega|^{\ad}.
 \]
\end{proof}

For a tame extension $(A,A^+) \to (B,B^+)$ of local Huber pairs we want to study the induced map on differentials
\[
 \Omega_A \longrightarrow \Omega_B.
\]
In particular, we want to compare the Kähler seminorms on both sides.
For this purpose we consider the factorization
\begin{equation} \label{factorization_Omega}
  \Omega_A \overset{\iota}{\longrightarrow} \Omega_A \otimes_A B \overset{\psi}{\longrightarrow} \Omega_B
\end{equation}
where $\Omega_A \otimes_A B$ is equipped with the projective tensor product seminorm:
For $\omega \in \Omega_A \otimes_A B$ its value is
\[
 |\omega| = \inf_{\omega = \sum \omega_i \otimes b_i} \max_i \{|\omega_i|_\Omega \cdot |b_i|\}.
\]
We showed in \cref{tensor_product_Huber_pairs} that~$\iota$ is an isometric embedding.

\begin{proposition} \label{isometry}
 Let $(B,B^+)/(A,A^+)$ be a tame extension of local Huber pairs.
 Then
 \[
  \Omega_A \otimes_A B \overset{\sim}{\longrightarrow} \Omega_B
 \]
 is an isometry (with respect to the K\"ahler seminorm).
\end{proposition}

\begin{proof}
 Consider the following map of distinguished triangles
 \[
  \begin{tikzcd}
   \LL^{\log}_{(A,A^+)} \otimes^h_{A^+} B^+	\ar[r]	\ar[d]	& \LL^{\log}_{(B,B^+)}	\ar[r]	\ar[d]	& \LL^{\log}_{(B,B^+)/(A,A^+)}	\ar[d]	\ar[r,"+1"]	& {} \\
   \LL_A \otimes^h_A B						\ar[r]			& \LL_B					\ar[r]			& \LL_{B/A}								\ar[r,"+1"]	& {}.
  \end{tikzcd}
 \]
 Since $B/A$ is \'etale, $\LL_{B/A} \cong 0$ (\cite{Ill71}, Proposition~3.1.1).
 Moreover, $(B,B^+)/(A,A^+)$ is tame, whence $\LL^{\log}_{(B,B^+)/(A,A^+)} \cong 0$ (see \cref{tame_local_Huber_pair_cotangent}).
 Furthermore, $B$ is flat over~$A$ and~$B^+$ is flat over~$A^+$ (see \cite{HueAd}, Proposition~11.8).
 Hence the derived tensor products are naive tensor products.
 We thus obtain a diagram
 \[
  \begin{tikzcd}
   \Omega^{\log}_{(A,A^+)} \otimes_{A^+} B^+	\ar[r,"\sim","\phi"']	\ar[d]	& \Omega^{\log}_{(B,B^+)}	\ar[d]	\\
   \Omega_A \otimes_A B							\ar[r,"\sim","\psi"']			& \Omega_B.
  \end{tikzcd}
 \]
 The vertical maps are localizations by \cref{localization_cotangent}.
 Hence, in order to show that~$\psi$ is an isometry, it suffices to show that~$\phi$ is an isometry.
 But the restriction of the K\"ahler seminorm to logarithmic differentials coincides with the adic seminorm (\cref{Kaehler_adic}) and for the adic seminorm~$\phi$ is an isometry by \cref{adic_seminorm_bc}.
\end{proof}

\begin{corollary} \label{Kaehler_equal_tame}
 Let $(A,A^+) \to (B,B^+)$ be a tame extension of local Huber pairs.
 We consider the corresponding map on differentials
 \[
  \Psi : \Omega_A \longrightarrow \Omega_B.
 \]
 Then for all $\omega \in \Omega_A$ we have
 \[
  |\omega|_{\Omega_A} = |\Psi(\omega)|_{\Omega_B}.
 \]
\end{corollary}

\begin{proof}
 The map~$\Psi$ factors as
 \[
  \Omega_A \overset{\iota}{\longrightarrow} \Omega_A \otimes_A B \overset{\psi}{\longrightarrow} \Omega_B.
 \]
 By \cref{tensor_product_Huber_pairs},~$\iota$ is an isometry and~$\psi$ is an isometry by \cref{isometry}.
\end{proof}

\subsection{The K\"ahler seminorm on adic spaces} \label{section_Kaehler_adic}

For a discretely ringed adic space~$\cX$ over $(k,k^+)$, a point $x \in \cX$, and an open neighborhood $\cU \subset \cX$ of~$x$
 we define \emph{the K\"ahler seminorm} $|-|_x$ on $\Omega_{\cX}(\cU)$ associated with~$x$ as follows.
For $\omega \in \Omega_{\cX}(\cU)$ let $\omega_x$ be the image of~$\omega$ in $\Omega_{\cX,x} = \Omega_{\cO_{\cX,x}}$.
Then
\[
 |\omega|_x := |\omega_x|_{\Omega},
\]
where $|-|_{\Omega}$ is the K\"ahler seminorm on $\Omega_{\cO_{\cX,x}}$ associated with the local Huber pair $(\cO_{\cX,x},\cO_{\cX,x}^+)$.

Recall from \cite{HueAd}, that an \'etale morphism $f : \cX \to \cY$ of adic spaces is \emph{strongly \'etale} at a point $x \in \cX$ if the residue field extension $k(x)|k(f(x))$ is unramified with respect to the valuation of $k(x)$ corresponding to~$x$.
Moreover, $f$ is \emph{tame} if $k(x)|k(f(x))$ is tamely ramified.
The tame (strongly \'etale) morphisms to~$\cX$ together with surjective families form a site~$\cX_t$ ($\cX_{\set}$), called the \emph{tame (strongly \'etale) site}.

\begin{definition}
 We define the subpresheaf~$\Omega^+$ of~$\Omega$ on~$\cX_t$ by
 \[
  \Omega^+(\cU) := \{\omega \in \Omega(\cU) \mid |\omega|_x \le 1~\forall x \in \cU\}.
 \]
\end{definition}
Notice that this construction is indeed functorial:
For $\cV \to \cU$ in~$\cX_t$, $\omega \in \Omega^+(\cU)$, and $x \in \cV$ we have
\[
 |(\omega|_{\cV})|_x = |\omega|_{f(x)} \le 1
\]
by \cref{Kaehler_equal_tame}.
By restriction, we obtain a presheaf on the topological space~$\cX$ and on the strongly \'etale site~$\cX_{\set}$, as well.
We denote all of these~$\Omega^+$.
Moreover, we set $\Omega^{n,+} := \bigwedge^n \Omega^+$.

\begin{proposition} \label{Kaeher_sheaf}
 The presheaf~$\Omega^{n+}$ is a sheaf on~$\cX_t$.
\end{proposition}

\begin{proof}
 For a covering $(\varphi_i:\cU_i \to \cU)$ in~$\cX_t$ consider the diagram
 \[
  \begin{tikzcd}
   0	\ar[r]	& \Omega^+(\cU)	\ar[r]	\ar[d,hook]	& \prod_i \Omega^+(\cU_i)	\ar[r]	\ar[d,hook]	& \prod_{ij} \Omega^+(\cU_i \times_{\cU} \cU_j)	\ar[d,hook]	\\
   0	\ar[r]	& \Omega(\cU)	\ar[r]				& \prod_i \Omega(\cU_i)		\ar[r]				& \prod_{ij} \Omega(\cU_i \times_{\cU} \cU_j).
  \end{tikzcd}
 \]
 The lower row is exact as~$\Omega$ is a sheaf.
 We have to show that the upper row is exact.
 Since $\Omega^+ \to \Omega$ is injective, exactness on the left hand side is clear.
 Let $(\omega_i)_i \in \prod_i \Omega^+(\cU_i)$ be such that
 \[
  \omega_i|_{\cU_i \times_{\cU} \cU_j} = \omega_j|_{\cU_i \times_{\cU} \cU_j} \quad \forall i,j.
 \]
 There is $\omega \in \Omega(\cU)$ such that $\omega|_{\cU_i} = \omega_i$ for all~$i$.
 In order to show that $\omega \in \Omega^+(\cU)$, take $x \in \cU$.
 Since $(\varphi_i : \cU_i \to \cU)$ is a covering, there is $i$ and $x_i \in \cU_i$ such that $\varphi_i(x_i) = x$.
 By \cref{Kaehler_equal_tame}
 \[
  |\omega|_x = |(\omega|_{\cU_i})|_{x_i} = |\omega_i|_{x_i} \le 1.
 \]
 Hence $\omega \in \Omega^+(\cU)$.
 The passage from~$\Omega$ to~$\Omega^n$ is straightforward.
\end{proof}

\begin{remark}
 For a morphism $\varphi: \cV \to \cU$ in~$\cX_{\et}$ that is not tame, $\omega \in \Omega(\cU)$, and $x \in \cV$, it is not true in general that $|(\omega|_{\cV})|_x = |\omega|_{\varphi(x)}$ (compare \cite{Tem16}, Theorem~5.6.4).
 We only have $|(\omega|_{\cV})|_x \le |\omega|_{\varphi(x)}$.
 So~$\Omega^+$ is a presheaf on the \'etale site but not necessarily a sheaf.
\end{remark}

\begin{proposition} \label{sheafification_logarithmic}
 Let $n \ge 1$ and~$\cX$ a (discretely ringed) adic space over $(k,k^+)$.
 As a sheaf on the topological space $\cX$, $\Omega^{n,+}$ is the sheafification of $\Omega^{n,\log}$.
 In particular, $\Omega^{n,+}$ is a subsheaf of~$\Omega^n$.
\end{proposition}

\begin{proof}
 Let us first show the proposition for $n = 1$.
 The homomorphism $\Omega^{\log} \to \Omega$ factors through~$\Omega^+$ as for an open $\cU \subseteq \cX$,
 $\omega \in \Omega^{\log}(\cU)$ and $x \in \cU$ we have
 \[
  \omega_x \in \Omega^{\log}_{(\cO_{\cX,x},\cO^+_{\cX,x})}
 \]
 and
 \[
  \big|\Omega^{\log}_{(\cO_{\cX,x},\cO^+_{\cX,x})}|_\Omega \le 1
 \]
 by \cref{unit_ball}.
 It thus suffices to show that for all $x \in \cX$ the induced homomorphism on stalks
 \[
  \Omega^{\log}_x \to \Omega^+_x
 \]
 is an isomorphism.
 This is precisely the assertion of \cref{unit_ball}.

 For $n \ge 1$ it is clear by  definition and from the result for $n = 1$ that the sheafification of $\Omega^{n,\log}$ is~$\Omega^{n,+}$.
 It then follows from \cref{Omega_injective} that the natural homomorphism $\Omega^{n,+} \to \Omega^n$ is injective.
\end{proof}

Note that the sheafification of~$\Omega^{n,\log}$ on the topological space~$\cX$ also provides the sheafification on the strongly \'etale and on the tame site as~$\Omega^{n,+}$ is a tame sheaf.

\section{Differentials on smooth adic spaces} \label{section_differentials_smooth}

\subsection{Setup}

Recall from \cite{Hu96}, Definition~1.6.5 that a morphism $\cX \to \cY$ of adic spaces is \emph{smooth} if it is locally of finite presentation and for every morphism $\Spa(A,A^+) \to \cY$ from an affinoid adic space and every ideal $I$ of~$A$ with~$I^2 = 0$, the homomorphism
\[
 \Hom_\cY(\Spa(A,A^+),\cX) \to \Hom_\cY(\Spa(A,A^+)/I,\cX)
\]
is surjective.

We fix a perfect field~$k$ and consider discretely ringed adic spaces over $\Spa(k,k)$.
For short we will speak of adic spaces over~$k$.

A pair of schemes $(X,\bar{X})$ is called \emph{log smooth} if $X$ is an open subscheme of~$\bar{X}$
 (we implicitly take the immersion $X \to \bar{X}$ as part of the datum)
 such that the associated log structure on~$\bar{X}$ is log smooth over~$k$.
We say that $X \to \bar{X}$ is a \emph{log smooth presentation} of an adic space~$\cX$ over~$k$ if $\cX = \Spa(X,\bar{X})$ and $(X,\bar{X})$ is log smooth.
In particular, if~$\cX$ has a log smooth presentation, it is smooth.
The converse direction only holds under the assumption that resolutions of singularities exist over~$k$.

For a morphism of schemes $X \to S$ such that $\Spa(X,S)$ is a smooth adic space over~$k$, we consider the following site $(X,S)_{\log}$:
The objects are finite disjoint unions of log smooth pairs $(Y,\bar{Y})$ fitting into a diagram
\[
 \begin{tikzcd}
  Y			\ar[r]	\ar[d]	& X	\ar[d]	\\
  \bar{Y}	\ar[r]			& S
 \end{tikzcd}
\]
such that $Y \to X$ is an open immersion and $\bar{Y} \to S$ is the normalization in~$Y$ of a scheme of finite type over~$S$.
The morphisms are compatible morphisms of pairs over $(X,S)$
 (but we do not require the associated morphism of log schemes to be log smooth).
If $(X,S)$ itself is log smooth, it is a final object of $(X,S)_{\log}$.
A morphism $(Y',\bar{Y}') \to (Y,\bar{Y})$ in $(X,S)_{\log}$ is called an \emph{open immersion} if the associated morphism of log schemes is an open immersion, i.e., $\bar{Y}' \to \bar{Y}$ is an open immersion and $Y' = Y \times_{\bar{Y}} \bar{Y}'$.
We define the coverings of $(X,S)_{\log}$ to be surjective families
\[
 ((Y_i,\bar{Y}_i) \to (Y,\bar{Y}))_{i \in I}
\]
of open immersions.
In other words, the topology is the Zariski topology on~$\bar{Y}$.

On $(X,S)_{\log}$ we consider the sheaf~$\Omega^{n,\log}$ of logarithmic differentials (compare \cite{Ogus18}, Theorem~1.2.4).
It is no coincidence that the symbol~$\Omega^{n,\log}$ is the same as for the presheaf of logarithmic differentials on the site of strict affinoids studied in \cref{section_presheaf_log_diff}.
In fact for a strict affinoid $\Spa(A,A^+)$ such that $(\Spec A,\Spec A^+)$ is log smooth, we have
\[
 \Omega^{n,\log}_{(A,A^+)} = \Omega^{n,\log}(\Spec A,\Spec A^+)
\]
by construction.
Because of this compatibility the use of~$\Omega^{n,\log}$ in both situations will not cause confusion.

For an object $(Y,\bar{Y})$ of $(X,S)_{\log}$ the induced morphism
\[
 \Spa(Y,\bar{Y}) \to \Spa(X,S)
\]
is an open immersion.
We thus obtain a morphism of sites
\[
 \ell: \Spa(X,S)_{\Top} \to (X,S)_{\log}.
\]
For log smooth Huber pairs $(A,A^+)$ \cref{unit_ball} provides functorial homomorphisms
\[
 \Omega^{n,\log}(\Spec A,\Spec A^+) \longrightarrow \Omega^{n,+}_{(A,A^+)} = \ell_*\Omega^{n,+}(\Spec A,\Spec A^+).
\]
Since the log smooth pairs of the form $(\Spec A,\Spec A^+)$ form a basis of the topology of $(X,S)_{\log}$ and both~$\Omega^{n,\log}$ and $\ell_*\Omega^{n,+}$ are sheaves on $(X,S)_{\log}$, the above homomorphism extends to a homomorphism of sheaves
\[
 \varphi: \Omega^{n,\log} \to \ell_* \Omega^{n,+}.
\]
Our goal is to prove that if $\Spa(X,S)$ is smooth,~$\varphi$ is an isomorphism.
Since we do not want to use resolution of singularities, the argument is somewhat intricate.
It is inspired by \cite{HKK17}.
However, we have adapted the constructions to our situation to produce a more streamlined argument.

%
%
%
%

\subsection{Unramified sheaves} \label{section_unramified_sheaves}

\begin{definition}
 We say that a morphism of schemes $Y \to Z$ is an \emph{isomorphism in codimension one} if there is an open subscheme $U \subseteq Z$ containing all points of codimension $\le 1$ such that the base change $Y \times_Z U \to U$ is an isomorphism.
 A morphism $(Y,\bar{Y}) \to (Z,\bar{Z})$ in $(X,S)_{\log}$ is an \emph{isomorphism in codimension one} if $\bar{Y} \to \bar{Z}$ is an isomorphism in codimension one and $Y = Z \times_{\bar{Y}} \bar{Z}$.
 In this case we write $(Y,\bar{Y}) \sim_1 (Z,\bar{Z})$.
\end{definition}

In a similar way as in \cite{Morel12}, Definition~2.1, we define unramified sheaves:

\begin{definition}
 A sheaf~$\cF$ on $(X,S)_{\log}$ is called \emph{unramified} if for any open immersion $(Y',\bar{Y}') \to (Y,\bar{Y})$ in $(X,S)_{\log}$ with dense image the restriction
 \[
  \cF(Y,\bar{Y}) \to \cF(Y',\bar{Y}')
 \]
 is injective and an isomorphism if $(Y',\bar{Y}') \sim_1 (Y,\bar{Y})$ .

 A presheaf~$\cG$ on $\Spa(X,S)$ is called \emph{unramified} if $\ell_*\cG$ is an unramified sheaf.
\end{definition}

\begin{lemma}
 Let~$\cF$ be an unramified sheaf on $(X,S)_{\log}$.
 If $(Y',\bar{Y}') \to (Y,\bar{Y})$ in $(X,S)_{\log}$ induces an isomorphism $\Spa(Y',\bar{Y}') \to \Spa(Y,\bar{Y})$,
  then the restriction
  \[
   \cF(Y,\bar{Y}) \to \cF(Y',\bar{Y}')
  \]
  is an isomorphism.
\end{lemma}

\begin{proof}
 The morphism $\Spa(Y',\bar{Y}') \to \Spa(Y,\bar{Y})$ is an isomorphism if and only if $Y' \cong Y$ and $\bar{Y}' \to \bar{Y}$ is proper birational.
 Since $(Y,\bar{Y})$ is log smooth, $\bar{Y}$ is normal.
 Hence, the exceptional locus of $\bar{Y}' \to \bar{Y}$ in~$\bar{Y}$ is of codimension $\ge 2$.
 In other words, its complement $\bar{V} \subset \bar{Y}$ contains all points of codimension $\le 1$.
 By construction $Y \subseteq \bar{V}$ and the open immersion $\bar{V} \to \bar{Y}$ lifts to an open immersion $\bar{V} \to \bar{Y}'$.
 We thus obtain a diagram
 \[
  \begin{tikzcd}
												& (Y',\bar{Y}')	\ar[dd]	\\
   (Y,\bar{V})	\ar[ur,open]	\ar[dr,open]	&	\\
												& (Y,\bar{Y}).
  \end{tikzcd}
 \]
 The diagonal arrows are open immersions with dense image and the image of the lower one contains all points of codimension $\le 1$.
 Applying~$\cF$ yields
 \[
  \begin{tikzcd}
					& \cF(Y',\bar{Y}')	\ar[dl,hook]	\\
   \cF(Y,\bar{V})	&	\\
					& \cF(Y,\bar{Y}).	\ar[ul,"\sim"']	\ar[uu]
  \end{tikzcd}
 \]
 Since~$\cF$ is unramified, the lower diagonal arrow is an isomorphism and the upper one is injective.
 Hence, the vertical arrow is an isomorphism (and the upper diagonal one as well).
\end{proof}

For an open subset~$\cU$ of $\Spa(X,S)$ we define the following full subcategory $\cU_{\log}$ of $(X,S)_{\log}$.
Its objects are the objects $(Z,\bar{Z})$ of $(X,S)_{\log}$ such that the morphism $\Spa(Z,\bar{Z}) \to \Spa(X,S)$ induced by the structure morphism factors through $\cU$.
Obviously, for $\cU' \subseteq \cU$ we have $\cU'_{\log} \subseteq \cU_{\log}$.

In case $\cU = \Spa(Y,\bar{Y})$, all objects $(Z,\bar{Z})$ of $(X,S)_{\log}$ with a morphism $(Z,\bar{Z}) \to (Y,\bar{Y})$ are in $\Spa(Y,\bar{Y})_{\log}$.
But $\Spa(Y,\bar{Y})_{\log}$ might be bigger.
For instance, if $(Y,\bar{Y}) \to (Z,\bar{Z})$ is a morphism in $(X,S)_{\log}$ such that $\bar{Y} \to \bar{Z}$ is proper and not an isomorphism and $Z = Y$, then $(Z,\bar{Z}) \in \Spa(Y,\bar{Y})_{\log}$ but there is no morphism $(Z,\bar{Z}) \to (Y,\bar{Y})$.
Only in the affine case we have the following lemma:

\begin{lemma} \label{final_object}
 Let $(A,A^+)$ be log smooth.
 Then $(\Spec A,\Spec A^+)$ is a final object of $\Spa(A,A^+)_{\log}$.
\end{lemma}

\begin{proof}
 Let $(Y,\bar{Y})$ be an object of $\Spa(A,A^+)_{\log}$ and set $\cY = \Spa(Y,\bar{Y})$.
 Then $\cO_\cY(\cY) = \cO_Y(Y)$ and $\cO_\cY^+(\cY) = \cO_{\bar{Y}}(\bar{Y})$.
 By \cite{Hu94}, Proposition~2.1 there is a natural isomorphism
 \[
  \Hom((A,A^+),(\cO_\cY(\cY),\cO_\cY^+(\cY))) \cong \Hom(\cY,\Spa(A,A^+)).
 \]
 We thus obtain ring homomorphisms $A \to \cO_Y(Y)$ and $A^+ \to \cO_{\bar{Y}}(\bar{Y})$.
 By functoriality they fit into a commutative diagram
 \[
  \begin{tikzcd}
   \cO_Y(Y)							& A		\ar[l]			\\
   \cO_{\bar{Y}}(\bar{Y})	\ar[u]	& A^+.	\ar[l]	\ar[u]
  \end{tikzcd}
 \]
 The characterization of morphisms to affine schemes by homomorphisms of global sections of the structure sheaves yields a commutative diagram of schemes
 \[
  \begin{tikzcd}
   Y		\ar[r]	\ar[d]	& \Spec A	\ar[d]	\\
   \bar{Y}	\ar[r]			& \Spec A^+.
  \end{tikzcd}
 \]
 This defines a morphism $(Y,\bar{Y}) \to (\Spec A,\Spec A^+)$ in $\Spa(A,A^+)_{\log}$.
\end{proof}

%

\begin{lemma} \label{strict_factorization}
 Let $\Spa(Y,\bar{Y}) \subset \Spa(X,S)$ be open coming from a diagram of schemes
 \[
  \begin{tikzcd}
   Y		\ar[r]	\ar[d]	& X			\ar[d]	\\
   \bar{Y}	\ar[r]			& S.
  \end{tikzcd}
 \]
 Moreover, let $(Z,\bar{Z}) \in \Spa(Y,\bar{Y})_{\log}$.
 Then there is an open subscheme $\bar{U} \subseteq \bar{Z}$ isomorphic in codimension one, containing~$Z$, and such that $(Z,\bar{U}) \to (X,S)$ factors through $(Y,\bar{Y})$.
\end{lemma}

\begin{proof}
 Replacing, if necessary, $\bar{Y}$ with a compactification of~$Y$ over~$\bar{Y}$, we may assume that $Y \to \bar{Y}$ is an open immersion with dense image.
 The morphism $\Spa(Z,\bar{Z}) \to \Spa(Y,\bar{Y})$ provides an open immersion $\varphi: Z \to Y$ and a birational map $\bar{\varphi}: \bar{Z} \to \bar{Y}$.
 Since~$\bar{Z}$ is normal, $\bar{\varphi}$ is defined over an open subscheme $\bar{U} \subseteq \bar{Z}$ containing all points of codimension $\le 1$.
 Moreover, we may assume that~$\bar{U}$ contains~$Z$.
 By construction $(Z,\bar{U}) \to (X,S)$ factors through $(Y,\bar{Y})$.
\end{proof}

We want to remind the reader of the concept of Riemann-Zariski morphisms (see \cite{HueSch20}).
A point~$x$ of an adic space~$\cX$ is called \emph{Riemann-Zariski}, if it has no nontrivial horizontal specialization.
A \emph{Riemann-Zariski morphism} is a morphism of adic spaces mapping Riemann-Zariski points to Riemann-Zariski points.
Let us now consider morphisms of adic spaces $\Spa(Y,T) \to \Spa(X,S)$ arising from diagrams of schemes
\[
 \begin{tikzcd}
  Y	\ar[r]	\ar[d]	& X	\ar[d]	\\
  T	\ar[r]			& S.
 \end{tikzcd}
\]
The above diagram is said to have \emph{universally closed diagonal} if the induced morphism $Y \to X \times_S T$ is universally closed.
In this case the morphism $\Spa(Y,T) \to \Spa(X,S)$ is Riemann-Zariski and the converse holds if~$Y$ is quasi-compact and all residue field extensions of $Y \to X$ are algebraic (see \cite{HueSch20}, Lemma~12.7).
In case~$S$ is integral, $Y$ is quasi-compact, and $X \to S$ and $Y \to X$ (and hence also $Y \to T$) are open immersions with dense image, being Riemann-Zariski is equivalent to $Y \cong X \times_S T$.

\begin{lemma} \label{find_open_codimension1}
 Let $(Y,\bar{Y})$ be in $(X,S)_{\log}$.
 Let $(\Spa(Y_i,\bar{Y}_i) \to \Spa(Y,\bar{Y}))_{i \in I}$ be a finite Riemann-Zariski covering coming from a diagram of schemes
 \[
  \begin{tikzcd}
   Y_i		\ar[r]	\ar[d]	& Y		\ar[d]	\\
   \bar{Y}_i	\ar[r]		& \bar{Y}
  \end{tikzcd}
 \]
 with universally closed diagonal.
 Moreover, we take $(Z,\bar{Z}) \in \Spa(Y,\bar{Y})_{\log}$.
 Then there is an open immersion of the form $(Z,\bar{U}) \to (Z,\bar{Z})$ which is an isomorphism in codimension one such that
 \begin{itemize}
  \item $(Z,\bar{U}) \to (X,S)$ factors through $(Y,\bar{Y})$ and
  \item setting $Z_i = Z \times_Y Y_i$ and $\bar{U}_i = \bar{U} \times_{\bar{Y}} \bar{Y}_i$, the family $((Z_i,\bar{U}_i) \to (Z,\bar{U}))_{i \in I}$ is a covering in $(X,S)_{\log}$.
 \end{itemize}
\end{lemma}

\begin{proof}
 Using \cref{strict_factorization} we find an open subscheme $\bar{V} \subseteq \bar{Z}$ containing~$Z$ and isomorphic in codimension one such that $(Z,\bar{V}) \to (X,S)$ factors through $(Y,\bar{Y})$.
 Set $Z_i = Z \times_Y Y_i = Z \cap Y_i$ and $\bar{V}_i = \bar{V}_i \times_{\bar{Y}} \bar{Y}_i$.
 Since~$\bar{V}$ is normal, for each~$i$ the morphism $\bar{V}_i \to \bar{V}$ is an open immersion when restricted to a suitable open subscheme of~$\bar{V}$ isomorphic in codimension one and containing~$Z$ (use the Riemann Zariski property).
 Denote by~$\bar{U}$ the intersection of all of these subschemes for all~$i$.
 Setting $\bar{U}_i = \bar{U} \times_{\bar{V}} \bar{V}_i$ we obtain a diagram
 \[
  \begin{tikzcd}
   (Z_i,\bar{U}_i)	\ar[r,open]				\ar[d]	& (Z_i,\bar{V}_i)	\ar[r]	\ar[d]					& (Y_i,\bar{Y}_i)	\ar[d]	\\
   (Z,\bar{U})		\ar[r,open,"\sim_1"']			& (Z,\bar{V})		\ar[r]	\ar[d,open,"\sim_1"]	& (Y,\bar{Y})	\\
													& (Z,\bar{Z}).
  \end{tikzcd}
 \]
 All required properties of~$\bar{U}$ are clear except maybe that $(\bar{U}_i \to \bar{U})_{i \in I}$ is a surjective family.
 But by construction the family $(\Spa(Z_i,\bar{U}_i) \to \Spa(Z,\bar{U}))_{i \in I}$ is the pullback of the covering $(\Spa(Y_i,\bar{Y}_i) \to \Spa(Y,\bar{Y}))_{i \in I}$ by $\Spa(Z,\bar{U}) \to \Spa(Y,\bar{Y})$.
 In particular, it is surjective.
 This implies that $(\bar{U}_i \to \bar{U})_{i \in I}$ has to be surjective.
\end{proof}

\begin{definition}
 For an unramified sheaf~$\cF$ on $(X,S)_{\log}$ we define a presheaf~$\cF_{\lim}$ on $\Spa(X,S)$ as follows:
 \[
  \cF_{\lim}(\cU) = \lim_{(Y,\bar{Y}) \in \cU_{\log}} \cF(Y,\bar{Y}).
 \]
\end{definition}

We want to emphasise that in the above definition we are taking a limit and not a colimit.
The presheaf $\cF_{\lim}$ is not related to the pullback $\ell^*\cF$.
Notice moreover, that the definition of~$\cF_{\lim}$ is indeed functorial:
For open subsets $\cU' \subseteq \cU$ in $\Spa(X,S)$ we need a restriction $\cF_{\lim}(\cU) \to \cF_{\lim}(\cU')$ in
\[
 \Hom(\cF_{\lim}(\cU),\cF_{\lim}(\cU')) = \lim_{(Y',\bar{Y}')} \colim_{(Y,\bar{Y})}\Hom(\cF(Y,\bar{Y}),\cF(Y',\bar{Y}')).
\]
In other words, we have to find for each $(Y',\bar{Y}')$ in $\cU'_{\log}$ a $(Y,\bar{Y})$ in $\cU_{\log}$ and define a homomorphism
 \[
  \cF(Y,\bar{Y}) \to \cF(Y',\bar{Y}').
 \]
Moreover, these homomorphisms need to be compatible.
But for given $(Y',\bar{Y}')$ we can just take $(Y,\bar{Y}) = (Y',\bar{Y}')$ and the identity homomorphism on $\cF(Y',\bar{Y}')$.
This is clearly functorial.

\begin{lemma} \label{restrictions_injective}
 Let~$\mathcal{F}$ be an unramified sheaf on $(X,S)_{\log}$.
 Then for all dense open subspaces $\cU' \subseteq \cU \subseteq \Spa(X,S)$ the restriction
 \[
  \cF_{\lim}(\cU) \to \cF_{\lim}(\cU')
 \]
 is injective.
\end{lemma}

\begin{proof}
 Suppose $s = (s_{(Y,\bar{Y})})_{(Y,\bar{Y})} \in \cF_{\lim}(\cU)$ maps to zero in $\cF_{\lim}(\cU')$.
 This means that $s_{(Y,\bar{Y})} = 0$ for all $(Y,\bar{Y}) \in \cU'_{\log}$.
 Take any $(Y,\bar{Y})$ in $\cU_{\log}$.
 We have to show that $s_{(Y,\bar{Y})} = 0$.
 Without loss of generality we may assume that $(Y,\bar{Y})$ is connected.
 There is a dense open $\bar{Y}' \subseteq \bar{Y}$ such that, setting $Y' = \bar{Y}' \cap Y$, the morphism $\Spa(Y',\bar{Y}') \to \Spa(X,S)$ factors through~$\cU'$.
 Then $\cF(Y,\bar{Y}) \to \cF(Y',\bar{Y}')$ is injective and $s_{(Y,\bar{Y})}|_{(Y',\bar{Y}')} = 0$, whence $s_{(Y,\bar{Y})} = 0$.
\end{proof}

\begin{proposition} \label{F_lim_sheaf}
 With the above notation $\cF_{\lim}$ is a sheaf on $\Spa(X,S)$.
\end{proposition}

\begin{proof}
 It suffices to show the sheaf condition for coverings of the form $(\varphi_i: \Spa(Y_i,\bar{Y}_i) \to \Spa(Y,\bar{Y}))_{i \in I}$ with finite index set~$I$ coming from diagrams
 \[
  \begin{tikzcd}
   Y_i			\ar[r,"\varphi"]		\ar[d]	& Y			\ar[d]	\\
   \bar{Y}_i	\ar[r,"\bar{\varphi}"]			& \bar{Y}.
  \end{tikzcd}
 \]
 Moreover, by \cite{HueSch20}, Lemma~12.10, every such covering has a refinement which is Riemann-Zariski.
 Therefore, we may assume that our covering is Riemann-Zariski.
 We need to show that the sequence
 \[
  \cF_{\lim}(\Spa(Y,\bar{Y})) \to \prod_i \cF_{\lim}(\Spa(Y_i,\bar{Y}_i)) \rightrightarrows \prod_{ij} \cF_{\lim}(\Spa(Y_i,\bar{Y}_i) \cap \Spa(Y_j,\bar{Y}_j))
 \]
 is exact.
 Exactness on the left is assured by \cref{restrictions_injective}.
 Suppose we are given $s_i = (s_{i,(Z_i,\bar{Z}_i)})_{(Z_i,\bar{Z}_i)}$ in $\cF_{\lim}(\Spa(Y_i,\bar{Y}_i))$ such that the restrictions of~$s_i$ and~$s_j$ to $\Spa(Y_i,\bar{Y}_i) \cap \Spa(Y_j,\bar{Y}_j)$ coincide.
 By definition this means that
 \[
  s_{i,(Z,\bar{Z})} = s_{j,(Z,\bar{Z})}
 \]
 for all $(Z,\bar{Z}) \in (\Spa(Y_i,\bar{Y}_i) \cap \Spa(Y_j,\bar{Y}_j))_{\log}$.

 We have to find $s \in \cF_{\lim}(\Spa(Y,\bar{Y}))$ with $s|_{\Spa(Y_i,\bar{Y}_i)} = s_i$ for all~$i$.
 Let $(Z,\bar{Z})$ be in $\Spa(Y,\bar{Y})_{\log}$.
 In the following we explain how to define $s_{(Z,\bar{Z})}$.
 \cref{find_open_codimension1} provides us with an open subscheme $\bar{U} \subseteq \bar{Z}$ isomorphic in codimension one and containing~$Z$ such that $(Z,\bar{U}) \to \Spa(X,S)$ factors through $(Y,\bar{Y})$ and $((Z_i,\bar{U}_i) \to (Z,\bar{U}))_{i \in I}$ is a covering in $(X,S)_{\log}$ (where $Z_i = Z \times_Y Y_i$ and $\bar{U}_i = \bar{U} \times_{\bar{Y}} \bar{Y}_i$).
 Since~$\bar{\cF}$ is a sheaf on $(X,S)_{\log}$, the sequence
 \[
  0 \to \cF(Z,\bar{U}) \to \prod_i \cF(Z_i,\bar{U}_i) \to \prod_{i,j} \cF(Z_i \cap Z_j,\bar{U}_i \cap \bar{U}_j)
 \]
 is exact.
 The sections $s_{i,(Z_i,\bar{U}_i)} \in \cF(Z_i,\bar{U}_i)$ coincide on the intersections $(Z_i \cap Z_j,\bar{U}_i \cap \bar{U}_j)$.
 They thus lift to a unique section~$s_{(Z,\bar{U})}$ of~$\cF(Z,\bar{U})$.
 We define $s_{(Z,\bar{Z})}$ to be the preimage of $s_{(Z,\bar{U})}$ under the isomorphism $\cF(Z,\bar{Z}) \to \cF(Z,\bar{U})$.
 It follows from the fact that~$\cF$ is unramified that the~$s_{(Z,\bar{Z})}$ are compatible and define an element of $\cF_{\lim}(\Spa(Y,\bar{Y}))$.
 We leave the details to the reader.

 Let us show that $s|_{\Spa(Y_i,\bar{Y}_i)} = s_i$.
 This is equivalent to showing that for all $(Z,\bar{Z}) \in \Spa(Y_i,\bar{Y}_i)_{\log}$ we have $s_{(Z,\bar{Z})} = s_{i,(Z,\bar{Z})}$.
 By unramifiedness we can check this equality after restricting to $(Z,\bar{U})$ for an open subscheme $\bar{U} \subseteq \bar{Z}$ isomorphic in codimension one and containing~$Z$.
 By \cref{find_open_codimension1} we may thus assume that $(Z,\bar{Z}) \to (X,S)$ factors through $(Y_i,\bar{Y}_i)$ and $((Z_j,\bar{Z}_j) \to (Z,\bar{Z}))_{j \in I}$ (for $Z_j = Z_i \times_{Y_i},Y_j$ and $\bar{Z}_j = \bar{Z}_i \times_{\bar{Y}_i} \bar{Y}_j$) is a covering in $\Spa(X,S)_{\log}$.
 By construction, $s_{(Z,\bar{Z})}$ is uniquely defined by the condition $s_{(Z,\bar{Z})}|_{(Z_j,\bar{Z}_j)} = s_{j,(Z_i,\bar{Z}_j)}$ for all $j \in I$.
 In particular, $s_{(Z,\bar{Z})}|_{(Z_i,\bar{Z}_i)} = s_{i,(Z_i,\bar{Z}_i)}$.
 But $\bigcup_i(Z_i,\bar{Z}_i) = (Z,\bar{Z})$, so $s_{(Z,\bar{Z})} = s_{i,(Z,\bar{Z})}$.
\end{proof}

\begin{lemma}
 $\Omega^{n,\log}$ is unramified.
\end{lemma}

\begin{proof}
 Theorem~38 in \cite{Mat70} says that a noetherian normal domain is the intersection of the localizations at its height one prime ideals.
 It follows from this that the sheaf~$\cO$ on $(X,S)_{\log}$ defined by
 \[
  (Y,\bar{Y}) \mapsto \cO_{\bar{Y}}(\bar{Y})
 \]
 is unramified.
 Since the objects of $(X,S)_{\log}$ are log smooth, $\Omega^{n,\log}$ is a locally free $\cO$-module.
 Hence, it is unramified as well.
\end{proof}

\subsection{The comparison theorem}

We have a natural map $\Omega^{n,\log} \to \Omega^{n,\log}_{\lim}$ of presheaves on the site of strict affinoids $\Spa(X,S)_{\straff}$.
Since~$\Omega^{n,\log}_{\lim}$ is a sheaf by \cref{F_lim_sheaf}, this map factors through the sheafification~$\Omega^{n,+}$ of~$\Omega^{n,\log}$:
\[
 \Omega^{n,\log} \to \Omega^{n,+} \to \Omega^{n,\log}_{\lim}.
\]

\begin{proposition} \label{comparison_straff}
 Let $(A,A^+)$ be log smooth.
 Then the natural homomorphism
 \[
  \Omega^{n,\log}_{(A,A^+)} \to \Omega^{n,+}(\Spa(A,A^+))
 \]
 is an isomorphism.
\end{proposition}

\begin{proof}
 Consider the chain of homomorphisms
 \[
  \Omega^{n,\log}_{(A,A^+)} \overset{\varphi_1}{\to} \Omega^{n,+}(\Spa(A,A^+)) \overset{\varphi_2}{\to} \Omega^{n,\log}_{\lim}(\Spa(A,A^+)) \overset{\varphi_3}{\to} \Omega^n_A.
 \]
 By \cref{final_object} we know that $(\Spec A,\Spec A^+)$ is a final object of $\Spa(A,A^+)_{\log}$.
 Hence, we can identify $\Omega^{n,\log}_{\lim}(\Spa(A,A^+))$ with $\Omega^{n,\log}_{(A,A^+)}$ and then $\varphi_2 \circ \varphi_1$ is the identity.
 Moreover, $\varphi_3 \circ \varphi_2$ is the natural inclusion.
 We obtain
 \[
  \begin{tikzcd}
   \Omega^{n,\log}_{(A,A^+)}	\ar[r,"\varphi_1"]	\ar[rr,bend left,"\id"]	& \Omega^{n,+}(\Spa(A,A^+))	\ar[r,"\varphi_2"]	\ar[rr,bend right,"\text{inclusion}"]	& \Omega^{n,\log}_{\lim}(\Spa(A,A^+))	\ar[r,"\varphi_3"]	& \Omega^n_A.
  \end{tikzcd}
 \]
 A diagram chase shows that~$\varphi_1$ and~$\varphi_2$ are isomorphisms.
\end{proof}

\begin{theorem} \label{main_theorem}
 Let $(Y,\bar{Y})$ be in $(X,S)_{\log}$.
 Then
 \[
  \Omega^{n,+}(\Spa(Y,\bar{Y})) \cong \Omega^{n,\log}(Y,\bar{Y}),
 \]
 where~$\Omega^{n,\log}$ denotes the sheaf of logarithmic differentials on $(X,S)_{\log}$.
\end{theorem}

\begin{proof}
 Consider the subcategory~$\cC$ of $(X,S)_{\log}$ of objects of the form $(\Spec A,\Spec A^+)$.
 It is a site with the induced topology and the topoi associated with~$\cC$ and $(X,S)_{\log}$ are equivalent.
 We consider the following morphism of sites
 \begin{align*}
  \pi^{\straff}: \Spa(X,S)_{\straff}	& \longrightarrow	 \cC	\\
  \Spa(A,A^+)							& \mapsfrom		 (\Spec A,\Spec A^+),
 \end{align*}
 It fits into the following commutative diagram of morphisms of sites
 \[
  \begin{tikzcd}
   \Spa(X,S)_{\Top}		\ar[r,"\pi"]	\ar[d,"\iota^{\straff}"]	& (X,S)_{\log}	\ar[d,"\iota^{\cC}"]	\\
   \Spa(X,S)_{\straff}	\ar[r,"\pi^{\straff}"]						& \cC.
  \end{tikzcd}
 \]
 It follows by construction that $\pi^{\straff}_* \iota^{\straff}_* \cF = \iota^{\cC}_* \pi_* \cF$ for any presheaf~$\cF$ on $\Spa(X,S)_{\Top}$.
 Applying~$\pi^{\straff}_*$ to the homomorphism $\Omega^{n,\log} \to \iota^{\straff}_*\Omega^{n,+}$ of presheaves on $\Spa(X,S)_{\straff}$, we obtain a homomorphism
 \[
  \pi^{\straff}_* \Omega^{n,\log} \to \iota^{\cC}_* \pi_* \Omega^{n,+}.
 \]
 Unraveling the definitions, we see that $\pi^{\straff}_* \Omega^{n,\log}$ equals $\iota^{\cC}_*\Omega^{n,\log}$ (where now~$\Omega^{n,\log}$ denotes the sheaf of logarithmic differentials on $(X,S)_{\log}$).
 By \cref{comparison_straff} the above homomorphism is an isomorphism.
 Since the topoi associated to~$\cC$ and $(X,S)_{\log}$ are equivalent, we obtain an isomorphism
 \[
  \Omega^{n,\log} \to \pi_* \Omega^{n,+}
 \]
 of sheaves on $(X,S)_{\log}$.
 Evaluating at an object $(Y,\bar{Y})$ in $(X,S)_{\log}$ yields the result.
\end{proof}

\bibliographystyle{../meinStil}
\bibliography{../citations}

\end{document}